\documentclass[10pt, colorlinks]{article}

\usepackage{latexsym, amssymb, amsmath, amsthm, amscd}
\usepackage[all,cmtip]{xy}
\usepackage{amsmath}
\usepackage{amsthm}
\usepackage{amssymb}
\usepackage{amsfonts}
\usepackage{amsrefs}
\usepackage{color,authblk}
\usepackage{tikz}
\usepackage{amscd} 
\usepackage{comment}
\usepackage{mathrsfs}
\usepackage{comment}
\usepackage[all]{xy}
\usepackage{graphicx}
\usepackage{authblk}
\usepackage{hyperref}
\usepackage{bbm}
\usepackage{fullpage}
\usepackage[all,cmtip]{xy}
\usepackage{relsize}
\usepackage{amsmath,amscd}
\usepackage{tikz-cd}
\usepackage{tikz}
\usetikzlibrary{matrix}
\usepackage[mathcal]{euscript}
\usepackage{txfonts}

\usepackage{hyperref}

\numberwithin{equation}{section} \DeclareMathSizes{2}{10}{12}{13}
\makeatletter
\newcommand*{\doublerightarrow}[2]{\mathrel{
		\settowidth{\@tempdima}{$\scriptstyle#1$}
		\settowidth{\@tempdimb}{$\scriptstyle#2$}
		\ifdim\@tempdimb>\@tempdima \@tempdima=\@tempdimb\fi
		\mathop{\vcenter{
				\offinterlineskip\ialign{\hbox to\dimexpr\@tempdima+1em{##}\cr
					\rightarrowfill\cr\noalign{\kern.5ex}
					\rightarrowfill\cr}}}\limits^{\!#1}_{\!#2}}}
\newcommand*{\triplerightarrow}[1]{\mathrel{
		\settowidth{\@tempdima}{$\scriptstyle#1$}
		\mathop{\vcenter{
				\offinterlineskip\ialign{\hbox to\dimexpr\@tempdima+1em{##}\cr
					\rightarrowfill\cr\noalign{\kern.5ex}
					\rightarrowfill\cr\noalign{\kern.5ex}
					\rightarrowfill\cr}}}\limits^{\!#1}}}
\makeatother

\newtheorem{thmi}{Theorem}

\newtheorem{thm}{Proposition}[section]
\newtheorem{Thm}[thm]{Theorem}
\newtheorem{rem}[thm]{Remark}

\newtheorem{cor}[thm]{Corollary}
\newtheorem{eg}[thm]{Example}
\newtheorem{lem}[thm]{Lemma}
\newtheorem{defn}[thm]{Definition}

\title{\Large Characterizations of higher derivations and higher differential torsion theories in Eilenberg-Moore categories of monads}

\author{Dipti Paik \footnote{Department of Mathematics, Indian Institute of Technology, Delhi. Email: paikdipti09@gmail.com} $\qquad$ Divya Ahuja \footnote{Department of Mathematics, The Institute of Mathematical Sciences, Chennai. Email: divyaahuja1428@gmail.com.} $\qquad$ Surjeet Kour\footnote{Department of Mathematics, Indian Institute of Technology, Delhi. Email: koursurjeet@gmail.com.} }

\date{}

\begin{document}
	
	\maketitle 
	\begin{abstract}
        Let \(T\) be a monad on a category \(\mathscr{C}\). In this paper, we introduce the notion of higher derivations on the monad \(T\) and characterize them in terms of ordinary derivations on $T$. We also define higher derivations on modules over the monad $T$ in the Eilenberg-Moore category \(EM_T\) and establish their characterization in a similar manner. We provide several examples that illustrate and support our results. Furthermore, we examine the conditions under which a torsion theory on \(EM_T\) is higher differential, and show that this holds if and only if every higher derivation on a module \(M \in EM_T\) extends uniquely to its module of quotients \(Q_{\tau}(M)\).

	\end{abstract}

	\medskip
	{MSC(2020) Subject Classification: 13N15, 16S90, 18C20, 18E40}
	
	
	\medskip
	\hypersetup{linktocpage}

	\smallskip
	\hypersetup{linktocpage}
	\section{Introduction}
Throughout this paper, we assume that $\mathfrak F$ is a field of characteristic zero, $A$ is an associative ring with unity, and ${}_{A}\mathrm{Mod}$ is the category of left $A$-modules.
 An additive map $\rho : A \to A$ is called a derivation on $A$ if it satisfies the Leibniz rule, i.e., $
\rho(ab) = \rho(a)b + a\rho(b)$ for all $a,b\in A.$ Given a derivation $\rho$ on $A$ and a left $A$-module $M$, an additive map $\varrho : M \to M$ is called a $\rho$-derivation on $M$ if $
\varrho(am) = a\varrho(m) + \rho(a)m$ for all $a \in A,\ m \in M$.
Higher derivations generalize this notion in a natural way. A family of additive maps $\Omega = \{\Omega_i : A \to A\}_{i=0}^n$ is called a higher derivation of order $n$ on $A$ if
$\Omega_k(ab) = \sum_{i=0}^k \Omega_i(a)\Omega_{k-i}(b)
$ for all $a,b \in A$ and $ k = 0,1,\ldots,n$.
Here, $\Omega_0$ is the identity map on $A$, and $\Omega_1$ coincides with an ordinary derivation. Let $\Omega$ be a higher derivation of order $n$ on $A$, and let $M \in {}_{A}\mathrm{Mod}$. A family of additive maps
$\Psi = \{\Psi_i : M \to M\}_{i=0}^n,$
with $\Psi_0$ the identity map on $M$, is called a higher $\Omega$-derivation of order $n$ on $M$ if
$
\Psi_k(am) = \sum_{i=0}^k \Omega_i(a)\Psi_{k-i}(m)
$  for all $a \in A,\ m \in M,$ and $k=0,\ldots,n.$
In particular, $\Psi_1$ is an ordinary $\Omega_1$-derivation on $M$.

\smallskip
The interaction between torsion theories and derivations has been studied extensively in the classical module theoretic setting.  For a torsion theory $\tau=(\mathscr T,\mathscr F)$ on the category of $A$-modules $_A\mathrm{Mod}$, it is known that $\tau$ is hereditary if and only if the module of quotients exists for every module $M$, and in this case the module of quotients $Q_{\tau}(M)$ is the $\tau$-injective envelope of the associated torsion-free module $M/M_{\tau}$, where $M_{\tau}$ denotes the $\tau$-torsion submodule of $M$ (see \cite{PB2}, \cite{OK}). Further, a hereditary torsion theory $\tau$ is called differential if for every $M\in~_A\mathrm{Mod}$ and
every $\rho$-derivation $\varrho$ on $M$, one has $\varrho(M_\tau) \subseteq M_\tau$ (see \cite{PB}). Differential torsion theories are significant as they allow the unique extension of $\rho$-derivation $\varrho$ on any object $M\in~_{A}\mathrm{Mod}$ to its module of quotients. 
 It was shown in \cite{Gol2} that if $\tau$ is a torsion theory on ${}_{A}\mathrm{Mod}$ and $M$ is $\tau$-torsion-free, then every $\rho$-derivation $\varrho$on $M$ extends to a $\rho$-derivation on the module of quotients $Q_\tau(M)$. The uniqueness of such an extension was later established by Bland in \cite{PB}. Lomp and van den Berg in \cite{LB} showed that all hereditary torsion theories in $_{A}\mathrm{Mod}$ are differential. In \cite{AB}, Banerjee extended this result for preadditive categories by proving that if $\rho$ is a derivation on a preadditive category $\mathcal{A},$ then every hereditary torsion theory on $_{\mathcal A}\mathrm{Mod}$ must be differential and every $\rho$-derivation on an  $\mathcal A$-module $\mathcal{M}$ can be extended uniquely to its module of quotients. In a recent work \cite{DA}, Ahuja and Kour used monad in place of the small preadditive category and proved that if $\delta: T\longrightarrow T$ is a derivation on a monad $T$, then every $\delta$-derivation on a $T$-module $M$ extends uniquely to a $\delta$-derivation on the module of quotients of $M$. 
	
	\smallskip
	In the context of higher derivations, Rim in \cite{RIM} examined the extension problem for higher antiderivations and showed that, under suitable conditions, such maps on a module $M$ extend to its module of quotients. When the ring involution is trivial, this reduces to Ribenboim’s notion of higher derivations~ (see \cite{RP80}), thereby generalizing the results of Golan~\cite{Gol2}. Bland in \cite{PPB} later characterized the existence and uniqueness of these extensions.

     \smallskip
     In this paper, we will generalize the theory of higher derivations by Bland on rings and modules in \cite{PPB} to the context of higher derivations on monads and on the Eilenberg-Moore category of modules over a monad.  We begin with a monad $(T,\theta,\zeta)$  on a category  $\mathscr C$. We first introduce the notion of higher derivation $\Delta$ of order $n>1$ on a monad $T$ (see Definition \ref{D2.4} and Definition \ref{D2.6}) in Section~\ref{section2}. Furthermore, when $\mathscr C$ is a $\mathfrak F$-linear category, we give a characterization of higher derivation on a monad $T$ in terms of ordinary derivations on $T$, thereby extending Mirzavaziri’s characterization for algebras \cite{MIR}. In Section~\ref{section3}, we study modules $M \in EM_T$ and introduce the
notion of a higher $\Delta$-derivation of order $n > 1$ on $M$. Motivated
by the work of Banerjee and Kour \cite[Theorem~5.5]{AK}, which shows that higher derivations on modules
over commutative rings are completely determined by sequences of
ordinary derivations, we provide a characterization of higher $\Delta$-derivation on $M$ in terms of ordinary derivations on $M$ when $\mathscr{C}$ is a $\mathfrak F$-linear category.

\smallskip 
From Section~\ref{section4} onwards, we consider $\mathscr C$ to be a Grothendieck category, i.e., an abelian category in which colimits exist, filtered colimits 
are exact, and $\mathscr C$ has a set of generators $\{G_i\}_{i\in\Lambda}$. Let $(T, \theta, \zeta, \Delta)$ be a higher differential monad of order $n$ on $\mathscr{C}$ such that $T$ is exact and preserves colimits, and let $\tau = (\mathscr{T}, \mathscr{F})$ be a hereditary torsion theory on
$EM_T$. It follows from \cite[Corollary~3.4]{LLV} that hereditary torsion
theories on $EM_T$ are in one-to-one correspondence with families of
Gabriel filters. Let $\mathscr{L}=\{\mathscr L_{TG_i}\}_{i\in \Lambda}$ be the corresponding family of Gabriel filters on $EM_T$. Then, we study the higher differential torsion theory on $EM_T$ and prove the following main result of Section \ref{section4}.
	\begin{thmi}(see Theorem \ref{T3.4})
		Let $(T, \theta, \zeta,\Delta)$ be a higher differential monad of order $n$ on $\mathscr C$ that is exact and preserves colimits, and $\tau$ be the hereditary torsion theory on $EM_{T}$. Then the following statements are equivalent.
		
		\smallskip
		i) Every Gabriel filter corresponding to $\tau$ is $\Delta$-invariant.
		
		\smallskip
		ii) $\tau$ is a higher differential torsion theory of order $n$.
		
		\smallskip
		iii) For any $(M,f_M) \in EM_T$, every higher $\Delta$-derivation $\{D_n\}$ of order $n$, every submodule $N \subseteq M_\tau$, and every $j \in \{0,\ldots,n\}$, there exists $K \in \mathcal{L}_{TG_i}$ such that
		$(\Delta_j)_{G_i}(K) \subseteq \ker(TG_i \to N).$
		
		\smallskip
		iv) For any $(M,f_M) \in EM_T$, every higher $\Delta$-derivation $\{D_n\}$ of order $n$ on $M$, and every $N \subseteq M_\tau$, there exists $K \in \mathcal{L}_{TG_i}$ such that 
		$(\Delta_j)_{G_i}(K) \subseteq \ker(TG_i \to N)$ for all $j = 0,1,\ldots,n$.
		\end{thmi}
	Finally in Section \ref{section5} we study higher $\Delta$-derivation on the module of quotients \(Q_\tau (M)\) of \(M\) with respect to the torsion theory \(\tau\) on \(EM_T\). Motivated by Rim’s definition of extensions of antiderivations~\cite{RIM}, we introduce the notion of a higher \(\Delta\)-derivation on \(Q_\tau (M)\) when \(M\) is torsion-free, and subsequently establish the following final main result of our paper.	
	\begin{thmi}(see Theorem \ref{T5.5})
		Let $\tau$ be a hereditary torsion theory on $EM_T$ and $(M,f_M)$ be any object in $EM_T$. Then for any higher $\Delta$-derivation $D: M\longrightarrow M$ of order $n$ can be extended to a unique higher $\Delta$-derivation $\widetilde{D}:Q_\tau (M)\longrightarrow Q_\tau (M)$ of order $n$ if and only if $\tau $ is a higher differential torsion theory of order $n$.
	\end{thmi}
	
	\section{Characterization of Higher Derivations on Monads}\label{section2}
	In this section, we present a characterization of higher derivations on a monad on a category  $\mathscr{C}$ in terms of ordinary derivations.
	We begin by recalling the definitions of monad, Eilenberg-Moore category of modules over a monad, and ordinary derivation on monad from \cite{DA}.
	
	\begin{defn}
		A monad on a category $\mathscr{C}$ is a triple $(T,\theta,\zeta)$ consisting of an endofunctor $T:\mathscr{C}\longrightarrow \mathscr{C}$ together with two natural transformations $\theta:T\circ T\longrightarrow T$ and $\zeta:1_\mathscr{C}\longrightarrow T$ such that for any object $M\in \mathscr{C}$, we have
		\begin{equation}\label{eq}
			\theta_{M} \circ \theta_{TM} = \theta_{M}\circ T\theta_{M} \quad \text{and} \quad \theta_{M}\circ T\zeta_{M} = 1_{TM} = \theta_{M} \circ \zeta_{TM}~,
		\end{equation}
		$i.e$, the following diagrams commute
		\begin{equation*}
			\begin{tikzcd}[row sep=3em, column sep=6em]
				T\circ T\circ TM\arrow[r, "\theta_{TM} ", shorten >=2pt, shorten <=2pt] \arrow[d, "T\theta_M"']  & T\circ TM \arrow[d, "\theta_M"] \\
				T\circ TM \arrow[r, "\theta_M"'] & TM
			\end{tikzcd}
			\hspace{1cm}and\hspace{1cm}
			\begin{tikzcd}[row sep=3em, column sep=6em]
				TM \arrow[r, "T\zeta_M", shorten >=2pt, shorten <=2pt] \arrow[d, "\zeta_{TM}"'] \arrow[rd, "1_{TM}", shorten >=4pt] & T\circ TM \arrow[d, "\theta_M"] \\
				T\circ TM \arrow[r, "\theta_M"'] & TM~.
			\end{tikzcd}
		\end{equation*}
		
		A monad morphism between two monads $(T,\theta,\zeta)$ and $(T',\theta',\zeta')$ is a natural transformation $\phi:T\longrightarrow T'$ such that the following equalities hold:
		\begin{equation*}
			\phi \circ \theta=\theta'\circ (\phi\ast \phi)\qquad \text{and}\qquad \phi\circ \zeta=\zeta'~.
		\end{equation*}
	\end{defn}
	\begin{defn}
		A module over a monad $(T,\theta,\zeta)$ is a pair $(M,f_M)$ consisting of an object $M\in \mathscr C$ and a morphism  $f_M:TM\longrightarrow M$ in $\mathscr C$ such that the following equalities hold:
		\begin{equation*}
			f_{M}\circ \theta_M=f_{M}\circ T(f_{M}) \qquad \text{and}\qquad f_{M}\circ \zeta_{M}=1_M~.
		\end{equation*}
		Furthermore, a $T$-module morphism between two $T$-modules $(M,f_M)$ and $(M',f_{M'})$ is a morphism $g:M\longrightarrow M'$ in $\mathscr C$ such that 
		 \begin{equation*}
			f_{M'}\circ Tg=g\circ f_M~.
		\end{equation*}
		
		We denote by $EM_T$ the Eilenberg-Moore category of modules over the monad $ T$. Given an object $M\in\mathscr C$, it is easy to see that $(TM, f_{TM}=\theta_{M}:T\circ TM\longrightarrow TM)\in EM_T$. Further, it is well known (see \cite{Mac}) that there exists a natural isomorphism
		\begin{equation}\label{adj}
			EM_T(TM,N)\cong \mathscr{C}(M,N)
		\end{equation}
		for $M\in \mathscr{C}$ and $N\in EM_T$.
		
	\end{defn} We now recall the following definition from \cite{DA}.
	\begin{defn}\cite[Definition 3.1]{DA} 
		Let $(T, \theta, \zeta)$ be a monad on $\mathscr C$. A natural transformation \(\delta : T \longrightarrow T\) is said to be an ordinary derivation on $T$ if the following diagram commutes
		\[
		\begin{tikzcd}[row sep=3em, column sep=6em]
			T\circ T \arrow[r, "1 \ast \delta + \delta \ast 1", shorten >=2pt, shorten <=2pt] \arrow[d, "\theta"'] & T\circ T \arrow[d, "\theta"] \\
			T \arrow[r, "\delta"'] & T~.
		\end{tikzcd}
		\]
		We continue to call the quadruple \((T, \theta, \zeta, \delta)\) a \emph{differential monad}.
	\end{defn}
	\noindent
	We now recall the definition of a higher derivation of order $n$ on the ring \( A \).\\
    \smallskip A higher derivation of order \( n \) on \( A \) is a collection of additive maps \( \Omega = \{\Omega_i : A \longrightarrow A\}_{i=0}^n \), where $\Omega_0$ is the identity mapping on $A$ and $\Omega_k$ satisfyies the following condition:
	\begin{equation}
		\begin{aligned}
			\Omega_k(ab)&=\Omega_k(a)\Omega_0(b)+\Omega_{k-1}(a)\Omega_1(b)+\ldots+\Omega(a)\Omega_k(b)\\&=\Omega_k(a)b+\Omega_{k-1}(a)\Omega_1(b)+\ldots+a\Omega_k(b)
		\end{aligned}
	\end{equation}
	for all $a,b\in A$ and for all $k\geq1$. We now show that the notion of higher derivations on a ring naturally extends to the higher derivations on a monad $T$. Moreover, we observe that the characterization of higher derivation over ring 
	continues to hold in the monadic setting (see \cite{MIR}).

	\begin{defn}\label{D2.4}
		Let $(T,\theta,\zeta)$ be a monad on a category $\mathscr C$. We say that a family $\Delta=\{\Delta_i: T\longrightarrow T\}_{i=0}^{n}$ of natural transformations where $\Delta_0$ is the identity natural transformation on $T$ is a higher derivation of order $n$ on $T$ if the following diagram commutes 
		\begin{equation}\label{eqE}
			\begin{tikzcd}[row sep=3.8em, column sep = 8em]
				T\circ T \arrow{r}{1\ast \Delta_i +\Delta_1\ast \Delta_{i-1}+\ldots+\Delta_i\ast 1}\arrow{d}{\theta}& T\circ T\arrow{d}{\theta} \\
				T \arrow{r}{\Delta_i}& T
			\end{tikzcd}
		\end{equation}
		for all $i=0,1,\ldots,n$.
		We call the quadruple $(T,\theta,\zeta,\Delta)$ a higher differential monad of order $n$. 
	\end{defn}
	\begin{rem}\label{RE}
		Let $(T,\theta,\zeta,\Delta)$ be a higher differential monad of order $n$ on 
		$\mathscr{C}$. 
		For $n=1$, $\Delta$ coincides with an ordinary derivation on $T$. 
		Moreover, for any $0 \leq k \leq n$, the collection 
		\[
		\bar{\Delta_k} = \{ \Delta_i : T \longrightarrow T \}_{i=0}^k
		\]
		forms a higher derivation of order $k$ on $T$. 
		Hence, every higher derivation of order $n$ gives rise to higher derivations 
		of order $k$ for all $1 \leq k\leq n$.
	\end{rem}
    \begin{defn}\label{D2.6}
         A family of natural transformations $\Delta=\{\Delta_i:T\longrightarrow T\}_{i=0}^\infty$ is a higher derivation on a monad $(T,\theta,\zeta)$  if it satisfies the following diagram
    \begin{equation}
        	\begin{tikzcd}[row sep=3.8em, column sep = 8em]
				T\circ T \arrow{r}{1\ast \Delta_i +\Delta_1\ast \Delta_{i-1}+\ldots+\Delta_i\ast 1}\arrow{d}{\theta}& T\circ T\arrow{d}{\theta} \\
				T \arrow{r}{\Delta_i}& T
			\end{tikzcd}
    \end{equation}
    for all $i\geq0$, where it is understood that $\Delta_0$ is the identity natural transformation on $T$. In this case we call $(T,\theta,\zeta,\Delta)$ a higher differential monad.
    \end{defn}
   We now give a characterization of higher derivation on monad.
	\begin{thm}\label{RRRR}
		Let  $\mathscr C$ be a $\mathfrak F$-linear category and $(T, \theta, \zeta,\Delta)$ be a higher differential monad on $\mathscr C$. Then there is a sequence $\{ \delta_n\}_{n=1}^\infty$ of ordinary derivations on $T$ such that $$(n+1)\Delta_{n+1}=\sum_{k=0}^n \delta_{k+1}\circ\Delta_{n-k}$$ for each $n\geq 0$.
	\end{thm}
	\begin{proof}
		We prove this by induction on $ n$.  For $n=0$, consider $\delta_1=\Delta_1$. Since $\Delta_1$ is an ordinary derivation on $T$, the result is true for $ n=0. $ We now assume that the result is true for all $k\leq n$ and we prove it for $k=n+1.$ Consider 
		\begin{equation}
			\delta_{n+1}=(n+1)\Delta_{n+1}-\sum_{k=0}^{n-1} \delta_{k+1}\circ\Delta_{n-k}.\label{Q}
		\end{equation} It is enough to show that $\delta_{n+1}$ is an ordinary derivation on $T$ i.e, the
		following diagram commutes           
		\begin{equation*}
			\begin{tikzcd}[row sep=3.8em, column sep = 7em]
				T\circ T \arrow{r}{1\ast \delta_{n+1}+\delta_{n+1}\ast 1 }\arrow{d}{\theta} &T\circ T \arrow{d}{\theta}\\
				T \arrow{r}{\delta_{n+1}}& T.~
			\end{tikzcd}
		\end{equation*} 
		We see that
		\begin{align}
			\notag \delta_{n+1}\circ \theta 
			&=\Big[(n+1)\Delta_{n+1}-\sum_{k=0}^{n-1} \delta_{k+1}\circ\Delta_{n-k}\Big]\circ \theta\tag{by (\ref{Q})}\\\notag &=(n+1)\Delta_{n+1}\circ \theta-\sum_{k=0}^{n-1} \delta_{k+1}\circ\Delta_{n-k}\circ \theta \\\notag
			&=(n+1)\Big[\theta\circ \Big(\sum_{k=0}^{n+1}\Delta_k \ast \Delta_{n+1-k} \Big)\Big]-\Big[\sum_{k=0}^{n-1}\delta_{k+1}\circ\theta\circ\Big(\sum_{l=0}^{n-k}\Delta_l \ast \Delta_{n-k-l}\Big)\Big]\tag{by (\ref{eqE})}\\\notag
			&= \theta\circ\sum_{k=0}^{n+1}(k+n+1-k)\Delta_k \ast \Delta_{n+1-k}-\sum_{k=0}^{n-1}\sum_{l=0}^{n-k}\delta_{k+1}\circ\theta\circ(\Delta_l \ast \Delta_{n-k-l})\\
			&= \theta\circ\sum_{k=0}^{n+1}(k+n+1-k)\Delta_k \ast \Delta_{n+1-k}-\sum_{k=0}^{n-1}\sum_{l=0}^{n-k}\theta\circ(1\ast \delta_{k+1}+ \delta_{k+1}\ast1)\circ(\Delta_l \ast \Delta_{n-k-l}).\label{X}
		\end{align} 
		We now split equation (\ref{X}) into the following two parts and compute them independently. We define \begin{equation}
			K=\theta\circ\sum_{k=0}^{n+1}k(\Delta_k \ast \Delta_{(n+1-k)})-\sum_{k=0}^{n-1}\sum_{l=0}^{n-k}\theta\circ(\delta_{k+1}\ast1)\circ(\Delta_l \ast \Delta_{(n-k-l)})
		\end{equation}
		and \begin{equation}
			L=\theta\circ\sum_{k=0}^{n+1}(n+1-k)(\Delta_k \ast \Delta_{(n+1-k)})-\sum_{k=0}^{n-1}\sum_{l=0}^{n-k}\theta\circ(1\ast \delta_{k+1})\circ(\Delta_l \ast \Delta_{(n-k-l)}).
		\end{equation}
		Then, $\delta_{n+1}\circ \theta =K+L$. We now show that
		$$K=\theta \circ (\delta_{n+1}\ast1)\qquad \textup{and} \qquad L=\theta \circ (1 \ast \delta_{n+1}).$$ To compute $K$, we first note that $0\leq k+l\leq n$ and $k\neq n$. We now set $r=l+k$. Then we obtain
		\begin{align*}
			K&=\theta\circ\sum_{k=0}^{n+1}k(\Delta_k \ast \Delta_{n+1-k})-\sum_{r=0}^{n}\sum_{k=0,k\neq n}^{r}\theta\circ(\delta_{k+1}\ast1)\circ(\Delta_{r-k} \ast \Delta_{n-r})\\ &=\theta\circ\sum_{k=0}^{n+1}k(\Delta_k \ast \Delta_{n+1-k})-\sum_{r=0}^{n-1}\sum_{k=0,k\neq n}^{r}\theta\circ(\delta_{k+1}\ast1)\circ(\Delta_{r-k} \ast \Delta_{n-r})-\sum_{k=0}^{n-1}\theta\circ(\delta_{k+1}\ast 1)\circ\Delta_{n-k}~,
		\end{align*}
		i.e,\begin{align*}
			K+\sum_{k=0}^{n-1}\theta\circ(\delta_{k+1}\ast 1)\circ\Delta_{n-k}=\theta\circ\sum_{k=0}^{n+1}k(\Delta_k \ast \Delta_{n+1-k})-\sum_{r=0}^{n-1}\sum_{k=0,k\neq n}^{r}\theta\circ(\delta_{k+1}\ast1)\circ(\Delta_{r-k} \ast \Delta_{n-r})~.
		\end{align*}
		Now on setting $k=t+1$ in the first summand of $K$, we obtain
		\begin{align}
			K+\sum_{k=0}^{n-1}\theta\circ(\delta_{k+1}\ast 1)\circ\Delta_{n-k} &=\theta\circ\sum_{t=0}^{n}(t+1)(\Delta_{t+1} \ast \Delta_{n-t})-\sum_{r=0}^{n-1}\sum_{k=0,k\neq n}^{r}\theta\circ(\delta_{k+1}\ast1)\circ(\Delta_{r-k} \ast \Delta_{n-r})\notag\\&=\theta\circ\sum_{t=0}^{n-1}(t+1)(\Delta_{t+1} \ast \Delta_{n-t})-\sum_{r=0}^{n-1}\sum_{k=0,k\neq n}^{r}\theta\circ(\delta_{k+1}\ast1)\circ(\Delta_{r-k} \ast \Delta_{n-r})\notag\\
			&\qquad +(n+1)\theta\circ(\Delta_{n+1}\ast 1).
		\end{align}
		By interchanging law of composition, we note that
		$ (\delta_{k+1}\ast1)\circ(\Delta_{r-k} \ast \Delta_{n-r})= (\delta_{k+1}\circ\Delta_{r-k})\ast (1\circ \Delta_{n-r}).$ Further, as \(0 \leq r+1 \leq n\), by induction we have $
		(r+1)\Delta_{r+1} = \sum_{\substack{k=0 \\ k \neq n}}^{r} \delta_{k+1} \circ \Delta_{r-k}.$ Therefore, we obtain
		\begin{align*}
			K+\sum_{k=0}^{n-1}\theta\circ(\delta_{k+1}\ast 1)\circ\Delta_{n-k} &=\theta\circ\sum_{r=0}^{n-1}\Big[(r+1)\Delta_{r+1} -\sum_{k=0,k\neq n}^{r}\delta_{k+1}\circ\Delta_{r-k} \Big]\ast \Delta_{n-r}+(n+1)\theta\circ(\Delta_{n+1}\ast 1)\notag\\&=(n+1)\theta\circ(\Delta_{n+1}\ast 1).
		\end{align*}
		Hence, we get \begin{align*}
			K&=(n+1)\theta\circ(\Delta_{n+1}\ast 1)-\sum_{k=0}^{n-1}\theta\circ(\delta_{k+1}\ast 1)\circ\Delta_{n-k}\\&=(n+1)\theta\circ(\Delta_{n+1}\ast 1)-\sum_{k=0}^{n-1}\theta\circ(\delta_{k+1}\circ \Delta_{n-k})\ast 1\tag{by Interchanging law of composition}\\&=\theta\circ\Big[(n+1)(\Delta_{n+1}\ast 1)-\sum_{k=0}^{n-1}(\delta_{k+1}\circ \Delta_{n-k})\ast 1\Big]\\&=\theta\circ\Big[\Big((n+1)\Delta_{n+1}-\sum_{k=0}^{n-1}(\delta_{k+1}\circ \Delta_{n-k})\Big)\ast 1\Big]\\&=\theta \circ(\delta_{n+1}\ast1)~.
		\end{align*}
		Similarly, one can show that $L=\theta\circ(1\ast\delta_{n+1}).$
		Therefore $$\delta_{n+1}\circ \theta =K+L=\theta \circ(1\ast \delta_{n+1}+\delta_{n+1}\ast1).$$ i.e, $\delta_{n+1}$ is a derivation on $T$. This completes the proof.
	\end{proof} 
    \begin{cor}
        Let  $\mathscr C$ be a $\mathfrak F$-linear category and $(T, \theta, \zeta,\Delta)$ be a higher differential monad of order $n$ on $\mathscr C$. Then there is a sequence $\{ \delta_i\}_{i=1}^n$ of ordinary derivations on $T$ such that $$(i+1)\Delta_{i+1}=\sum_{k=0}^i \delta_{k+1}\circ\Delta_{i-k}$$ for each $0\leq i<n $.
    \end{cor}
    \begin{eg}\label{2}
       Let $(T, \theta, \zeta,\delta)$ be a differential monad on a $\mathfrak F$-linear category $\mathscr{C}$. Then the collection $\Delta=\{\Delta_n=\frac{\delta^n}{n!}\}_{n=1}^\infty$ with $\Delta_0=Id_T$ is a higher derivation on $T$.
    \end{eg}
    \begin{eg}
     Let $(T, \theta, \zeta,\Delta)$ be a higher differential monad of order $n$ on $\mathscr{C}$. For any $1\leq k\leq n$, consider the following sequence of natural transformations
    
      \[ \Delta'_t = \begin{cases}
          Id_T & \text{if } t=0, \\
          0 & \text{if } k\nmid t, \\
          \Delta_s & \text{if } t=ks. 
       \end{cases}
    \]
    Then $\Delta'=\{\Delta_t\}_{t=0}^n$ is a higher derivation of order $n$ on $T$.
    \end{eg}
    \begin{eg}\label{3}
       Let $(T, \theta, \zeta)$ be a monad on $\mathscr{C}$ . We know from \cite[Examples 3.2]{DA} that for a fix natural transformation $\mu:1_\mathscr{C}\longrightarrow T$, $\delta=\theta\circ (\mu_T-T\mu)$ is an an ordinadry derivation on $T$. We now define
       
      \begin{equation}\label{BRU} \Delta_n = \begin{cases} 
      Id & \text{for} ~n=0,\\
          \delta &  \text{for} ~n=1,\\
          \Delta_n= \theta^{n-2}\circ(\mu^{n-1}\ast\delta)& \text{for all } n\geq 2,
       \end{cases}
\end{equation}
       where $\theta^{n-2}_M=\theta_M\circ\theta_{TM}\circ\ldots\theta_{T^{n-2}M}$ and $\mu^{n-1}=\underbrace{\mu\ast\mu\ast\ldots\ast\mu}_{n-1}$.\\
       
       Note that, for all $n\geq2$, we have
       \begin{align}
       (\Delta_n)_{M}
          &=\theta^{n-2}_M\circ(\mu^{n-1}\ast\delta)_M\notag\\
           &=\theta_M\circ\theta_{TM}\circ\ldots\circ \theta_{T^{n-2}M}\circ (\mu^{n-1}\ast\delta)_M\notag\\
           &=\theta_M\circ T\theta_{M}\circ\ldots \circ T\theta_{T^{n-3}M}\circ (\mu\ast(\mu^{n-2}\ast\delta))_M\tag{$\theta_M\circ\theta_{TM}=\theta_M\circ T\theta_{M}$}\\
           &=\theta_M\circ T(\theta_{M}\circ\ldots \circ\theta_{T^{n-3}M})\circ \mu_{T^{n-1}M}\circ(\mu^{n-2}\ast\delta)_M\tag{taking $\nu_M=\theta_{M}\circ\ldots \circ\theta_{T^{n-3}M}:T^{n-1}M\longrightarrow T^2M$}\\
           &=\theta_M\circ\mu_{TM}\circ(\theta_M\circ\theta_{TM}\circ\ldots\circ\theta_{T^{n-3}M})\circ(\mu^{n-2}\ast\delta)_M\tag{$T\nu_M\circ \mu_{T^{n-1}M}=\mu_{TM}\circ \nu_M$}\\
           &=\theta_M\circ\mu_{TM}\circ(\Delta_{n-1})_M,
        \end{align}
        where the last equality follows from the fact that $\mu$ is a natural transformation. Therefore, we can write $ \Delta_n= \theta\circ\mu_T\circ\Delta_{n-1}$.

        \smallskip
       We now show that $\Delta=\{\Delta_n:T\longrightarrow T\}_{n\geq 0}$  forms a higher derivation on $T$. For this, it is enough to show that the following diagram
       \begin{equation}\label{eqHUH}
			\begin{tikzcd}[row sep=3.8em, column sep = 8em]
				T\circ T \arrow{r}{1\ast \Delta_n +\Delta_1\ast \Delta_{n-1}+\ldots+\Delta_n\ast 1}\arrow{d}{\theta}& T\circ T\arrow{d}{\theta} \\
				T \arrow{r}{\Delta_n}& T
			\end{tikzcd}
		\end{equation} 
        commutes for all $n.$\\
        For $n=0$ the result is trivially true and for $n=1$ the result follows directly from \cite[Examples 3.2]{DA}. We now assume that the statement is true for $k\leq n-1$ and we will prove it for $k=n$. We first observe that any $1\leq t\leq n-1$, we have 
        \begin{align}
        \theta_M\circ\mu_{TM}\circ\theta_M\circ(\Delta_t\ast \Delta_{n-t})_M
        &=\theta_M\circ\mu_{TM}\circ\theta_M\circ(\Delta_t)_{TM}\circ T(\Delta_{n-t})_M\notag\\
        &=\theta_M\circ T\theta_{M}\circ \mu_{TTM}\circ(\Delta_t)_{TM}\circ T(\Delta_{n-t})_M\tag{$\mu_{TM}\circ\theta_M=T\theta_M\circ\mu_{TTM}$}\\
        &=\theta_M\circ \theta_{TM}\circ \mu_{TTM}\circ(\Delta_t)_{TM}\circ T(\Delta_{n-t})_M\tag{$\theta_M\circ T\theta_{M}=\theta_M\circ \theta_{TM}$}\\
        &=\theta_M\circ (\Delta_{t+1})_{TM}\circ T(\Delta_{n-t})_M\tag{by \eqref{BRU}}\\
        &=\theta_M\circ (\Delta_{t+1}\ast \Delta_{n-t})_M.\label{WE}
        \end{align}
        Also, we see that \begin{align}
            &\theta_M\circ T\theta_M\circ\mu_{TTM}\circ T(\Delta_{n-1})_M\notag\\
            &= \theta_M\circ T\theta_M\circ\mu_{TTM}\circ T(\Delta_{n-1})_M-\theta_M\circ \theta_{TM}\circ T\mu_{TM}\circ T(\Delta_{n-1})_M +\theta_M\circ \theta_{TM}\circ T\mu_{TM}\circ T(\Delta_{n-1})_M\notag\\
            &=\theta_M\circ \theta_{TM}\circ(\mu_{TTM}-T\mu_{TM})\circ T(\Delta_{n-1})_M+\theta_M\circ T\theta_{M}\circ T\mu_{TM}\circ T(\Delta_{n-1})_M\notag\\
            &=\theta_M \circ(\Delta_1)_{TM}\circ T(\Delta_{n-1})_M+\theta_M\circ T\theta_{M}\circ T\mu_{TM}\circ T(\Delta_{n-1})_M.\label{33}
        \end{align}
Finally for $k=n$, we get
        \begin{align*}
           (\Delta_n\circ \theta)_M 
            &=(\theta\circ\mu_T\circ\Delta_{n-1})_M\circ \theta_M\tag{by\eqref{BRU}}\\
            &=\theta_M\circ\mu_{TM}\circ(\Delta_{n-1})_M\circ\theta_M\\
            &=\theta_M\circ\mu_{TM}\circ\theta_M\circ(1\ast \Delta_{n-1} +\Delta_1\ast \Delta_{n-2}+\ldots+\Delta_{n-1}\ast 1)_M\tag{by induction}\\
            &=\theta_M\circ\mu_{TM}\circ\theta_M\circ(1\ast \Delta_{n-1})_M +\theta_M\circ\mu_{TM}\circ\theta_M\circ(\Delta_1\ast \Delta_{n-2})_M+\ldots+\theta_M\circ\mu_{TM}\circ\theta_M\circ(\Delta_{n-1}\ast 1)_M
            \end{align*}
            \begin{align*}
            &=\theta_M\circ\mu_{TM}\circ\theta_M\circ(1\ast \Delta_{n-1})_M +\theta_M\circ(\Delta_2\ast \Delta_{n-2})_M+\ldots+\theta_M\circ(\Delta_{n}\ast 1)_M\tag{by \eqref{WE}}\\
           &=\theta_M\circ \mu_{TM}\circ\theta_M  \circ T(\Delta_{n-1})_M +\theta_M\circ(\Delta_2\ast \Delta_{n-2})_M+\ldots+\theta_M\circ(\Delta_{n}\ast 1)_M\\
           &=\theta_M\circ T\theta_M\circ\mu_{TTM}\circ T(\Delta_{n-1})_M +\theta_M\circ(\Delta_2\ast \Delta_{n-2})_M+\ldots+\theta_M\circ(\Delta_{n}\ast 1)_M\tag{$\mu_{TM}\circ\theta_M=T\theta_M\circ\mu_{TTM}$}\\
           &=\theta_M\circ (\Delta_1)_{TM}\circ T(\Delta_{n-1})_M+\theta_M\circ \theta_{TM}\circ T\mu_{TM}\circ T(\Delta_{n-1})_M +\theta_M\circ(\Delta_2\ast \Delta_{n-2})_M+\ldots+\theta_M\circ(\Delta_{n}\ast 1)_M\tag{by\ref{33}}\\
           &=\theta_M\circ (\Delta_1\ast\Delta_{n-1})_M+\theta_M\circ T\theta_{M}\circ T\mu_{TM}\circ T(\Delta_{n-1})_M +\theta_M\circ(\Delta_2\ast \Delta_{n-2})_M+\ldots+\theta_M\circ(\Delta_{n}\ast 1)_M\tag{$\theta_M\circ T\theta_{M}=\theta_M\circ \theta_{TM}$}\\
           &=\theta_M\circ (\Delta_1\ast\Delta_{n-1})_M+\theta_M\circ T(\Delta_{n})_M +\theta_M\circ(\Delta_2\ast \Delta_{n-2})_M+\ldots+\theta_M\circ(\Delta_{n}\ast 1)_M\tag{by \eqref{BRU}}\\
           &=\theta_M\circ (\Delta_1\ast\Delta_{n-1})_M+\theta_M\circ (1\ast\Delta_{n})_M+\theta_M\circ(\Delta_2\ast \Delta_{n-2})_M+\ldots+\theta_M\circ(\Delta_{n}\ast 1)_M\\
            &=\theta_M\circ (1\ast \Delta_{n} +\Delta_1\ast \Delta_{n-1}+\ldots+\Delta_{n}\ast 1)_M.
        \end{align*}
        This completes the proof.
    \end{eg}
	\section{Characterization of higher derivations on a module over a monad}\label{section3}
	In this section, we introduce the notion of higher derivation on modules over a monad in the Eilenberg-Moore category and establish a similar characterization of higher derivation on modules over a monad in terms of ordinary derivations on module over monad as discussed in \cite[Theorem 5.5]{AK}. We start by recalling the definition of ordinary derivation on a module over a monad from \cite{DA}.
	\begin{defn}\cite[Definition 3.5]{DA}
		Let $(T,\theta,\zeta,\delta)$ be a differential monad on a category $\mathscr{C}$ and $(M,f_M)$ be an object in $EM_T$. Then,  a morphism $ d:M\longrightarrow M$ in $\mathscr{C}$ is called  a $\delta$-derivation on $M$ if the following diagram commutes:
        \begin{equation}\label{KKKK}
        \begin{tikzcd}[row sep=3em, column sep=5em]
            TM \arrow[r, "Td+\delta_M"] \arrow[d, "f_M"'] 
            & TM \arrow[d, "f_M"] \\
            M \arrow[r, "d"'] 
            & M
        \end{tikzcd}
        \end{equation}
	\end{defn}
	Recall that given a higher derivation $\Omega$ of order $n$ on a ring $A$, a higher $\Omega$-derivation of order $n$ on an $A$-module $M$ is a family of morphisms $\Psi:=\{\Psi_k:M\longrightarrow M\}_{k=0}^n$ , where $\Psi_0$ is the identity mapping on $M$ and $\Psi_k$ satisfies the following condition:
	\begin{equation}
		\begin{aligned}
			\Psi_k(am)&=\Omega_0(a)\Psi_k(m)+\Omega_1(a)\Psi_{k-1}(m)+\ldots+\Omega_k(a)\Psi_0(m)\\&=a\Psi_k(m)+\Omega_1(a)\Psi_{k-1}(m)+\ldots+\Omega_k(a)m
		\end{aligned}
	\end{equation}
	for all $a\in A$, $m\in M$ and for all $k\geq 0$ (see \cite{PPB}). We now show that the notion of higher derivations on $A$-modules naturally extend to the higher derivations on modules over a monad.
	\begin{defn}\label{DM}
		Let $(T, \theta, \zeta,\Delta)$ be a higher differential monad of order $n$ on $\mathscr C$ and let $(M,f_M)$ be an object in $EM_{T}$. Then, we say that a family $D=\{D_i: M\longrightarrow M\}_{i=0}^{n}$ of morphisms on $M$ in $\mathscr C$ is a higher $\Delta$-derivation of order $n$ on $M$ if the following diagram commutes 
		\begin{equation}\label{eq33}
			\begin{tikzcd}[row sep=4.8em, column sep = 16em]
				TM \arrow{r}{TD_i+TD_{i-1}\circ (\Delta_{1})_{M}+\ldots+TD_1\circ (\Delta_{i-1})_{M}+(\Delta_i)_{M}}\arrow{d}{f_M} &TM \arrow{d}{f_M}\\
				M \arrow{r}{D_i}& M
			\end{tikzcd}
		\end{equation} 
		for all $i=0,1,\ldots,n$, where it is understood that $D_0=Id_M$, the identity morphism on $M$ in $\mathscr{C}$ .
	\end{defn}
	\begin{rem}\label{51}
		Let $(M,f_M)$ be a module over a higher differential monad $(T,\theta,\zeta,\Delta)$ of order $n$, and  $D = \{ D_i : M \longrightarrow M \}_{i=0}^n$ be a higher $\Delta$-derivation . For $n=1$, $D$ is an ordinary derivation on $M$ relative to $\Delta_1$. 
		Furthermore, for any $0 < k \leq n$, the collection
		\[
		\Bar{D_k} = \{ D_i : M \longrightarrow M \}_{i=0}^k
		\]
		defines a higher derivation of order $k$ on $M$ with respect to $\Bar{\Delta_k}$ as defined in Remark \ref{RE}. 
		Thus, every higher derivation of order $n$ on $M$ defines higher derivations 
		of order $k$ for all $0<k \leq n$.
	\end{rem}
    \smallskip 
     \begin{defn}
     Let $(M,f_M)$ be a module over a higher differential monad $(T,\theta,\zeta,\Delta)$. Then we say that a family of morphisms $D=\{D_i: M\longrightarrow M\}_{i=0}^{\infty}$ in $\mathscr C$ is a higher $\Delta$-derivation on $M$ if the following diagram  
		\begin{equation}\label{eq3.1}
			\begin{tikzcd}[row sep=4.8em, column sep = 16em]
				TM \arrow{r}{TD_i+TD_{i-1}\circ {(\Delta_{1})}_{M}+\ldots+TD_1\circ (\Delta_{{i-1}})_{M}+(\Delta_i)_{M}}\arrow{d}{f_M} &TM \arrow{d}{f_M}\\
				M \arrow{r}{D_i}& M
			\end{tikzcd}
		\end{equation} 
	commutes for all $i\geq 0$ with the understanding that $D_0$ is the identity morphism on $M$ in $\mathscr{C}$.
    \end{defn}
        
We now give a characterization of higher derivation on a module over monad.
	\begin{thm}
		Let  $\mathscr C$ be a $\mathfrak F$-linear category and $(T, \theta, \zeta,\Delta)$ be a higher differential monad on $\mathscr C$, where $\Delta$ can be determined by the family $\delta=\{\delta_i:T\longrightarrow T\}_{i=1}^\infty$ of ordinary derivation on $T$. Let $(M,f_M)$ be an object in $EM_T$ that is equipped with a higher $\Delta$-derivation $D$. Then there is a sequence $\{ d_i:M\longrightarrow M\}_{i=0}^\infty$ of ordinary derivations on $M$  such that $$(n+1)D_{n+1}=\sum_{k=0}^n d_{k+1}\circ D_{n-k}$$ for each $n\geq 0$.
	\end{thm}
	\begin{proof}
		We prove this result by induction on $n$. Since $D_1$ is an ordinary $\delta_1$-derivation on $M$, we have \[
        D_1\circ f_M=f_M\circ (TD_1+\delta_{1_M})=f_M\circ( Td_1+\delta_{1_M}).
        \] Set $d_1=D_1$. Then $d_1$ is a $\delta_1$-derivation. Hence the result is true for $n=0$. We now assume that the result is true for all $k\leq n$ and we prove it for $n+1$. We set $$d_{n+1}=(n+1)D_{n+1}-\sum_{k=0}^{n-1}d_{k+1}\circ D_{n-k}.$$ We proceed in a manner similar to the proof of Proposition~\ref{RRRR} to show that $d_{n+1}$ is an ordinary $\delta_{n+1}$- derivation on $M$ i.e, the following diagram commutes
        \begin{equation}\label{KKK}
        \begin{tikzcd}[row sep=3em, column sep=5.5em]
           TM \arrow[r, "Td_{n+1}+(\delta_{n+1})_{M}"] \arrow[d, "f_{M}"'] 
           & TM \arrow[d, "f_{M}"] \\
           M \arrow[r, "d_{n+1}"'] 
           & M
        \end{tikzcd}
        \end{equation}
		We see that \begin{align}
			d_{n+1}\circ f_M
			&=\Big((n+1)D_{n+1}-\sum_{k=0}^{n-1}d_{k+1}\circ D_{n-k}\Big)\circ f_M\notag\\
			&=(n+1)D_{n+1}\circ f_M-\sum_{k=0}^{n-1}d_{k+1}\circ D_{n-k}\circ f_M\notag
            \end{align}
            \begin{align}
			&=(n+1)f_M\circ \sum_{k=0}^{n+1}TD_k\circ {(\Delta_{n+1-k})_M}-\sum_{k=0}^{n-1}d_{k+1}\circ f_M \circ \sum_{t=0}^{n-k}TD_t\circ (\Delta_{n-k-t})_M\tag{by Remark \ref{51}}\\
			&=(n+1)f_M\circ \sum_{k=0}^{n+1}TD_k\circ (\Delta_{n+1-k})_M-\sum_{k=0}^{n-1}\sum_{t=0}^{n-k}d_{k+1} \circ f_M \circ TD_t\circ (\Delta_{n-k-t})_M\notag\\
			&=(k+n+1-k)f_M\circ \sum_{k=0}^{n+1}TD_k\circ (\Delta_{n+1-k})_M-
			\sum_{k=0}^{n-1}\sum_{t=0}^{n-k}f_M\circ \Big(Td_{k+1}+\delta_{k+1}\Big)\circ TD_t\circ (\Delta_{n-k-t})_M.\label{eq3.4*}
		\end{align}
		We split \eqref{eq3.4*} in the following two parts:
		\begin{equation}
			L_1=kf_M\circ \sum_{k=0}^{n+1}TD_k\circ (\Delta_{n+1-k)})_M-\sum_{k=0}^{n-1}\sum_{t=0}^{n-k}f_M\circ Td_{k+1}\circ TD_t\circ (\Delta_{n-k-t})_M 
		\end{equation}
		and 
		\begin{equation}
			L_2=(n+1-k)f_M\circ \sum_{k=0}^{n+1}TD_k\circ (\Delta_{n+1-k})_M-\sum_{k=0}^{n-1}\sum_{t=0}^{n-k}f_M\circ \delta_{k+1}\circ TD_t\circ (\Delta_{n-k-t})_M.
		\end{equation}
		Note that $0\leq t+k\leq n$ and $k\neq n$. On setting $r=t+k$ in the 2nd summand of $L_1 $, we obtain
		\begin{align}
			L_1&=f_M\circ \sum_{k=0}^{n+1}kTD_k\circ (\Delta_{n+1-k})_M-\sum_{r=0}^{n}\sum_{k=0}^{r}f_M\circ Td_{k+1}\circ TD_{r-k}\circ (\Delta_{n-r})_M\notag\\
			&=(n+1)f_M\circ TD_{n+1}+f_M\circ \sum_{k=0}^{n}kTD_k\circ (\Delta_{n+1-k})_M-\sum_{k=0}^{n-1}f_M\circ Td_{k+1}\circ TD_{n-k}-\sum_{r=0}^{n-1}\sum_{k=0}^{r}f_M\circ Td_{k+1}\circ TD_{r-k}\circ (\Delta_{n-r})_M.
		\end{align}
	Further, taking $k=l+1$ in the 2nd summand of $L_1$, we obtain
		
		\smallskip
		\begin{align*}
			L_1&=(n+1)f_M\circ TD_{n+1}+f_M\circ \sum_{l=0}^{n-1}(l+1)TD_{l+1}\circ (\Delta_{n-l})_M-\sum_{k=0}^{n-1}f_M\circ Td_{k+1}\circ TD_{n-k}-\sum_{r=0}^{n-1}\sum_{k=0}^{r}f_M\circ Td_{k+1}\circ TD_{r-k}\circ (\Delta_{n-r})_M\notag\\
			&=f_M\Big((n+1)TD_{n+1}-\sum_{k=0}^{n-1}Td_{k+1}\circ TD_{n-k}\Big)+f_M\circ \sum_{l=0}^{n-1}(l+1)TD_{l+1}\circ (\Delta_{n-l})_M-\sum_{r=0}^{n-1}\sum_{k=0}^{r}f_M\circ Td_{k+1}\circ TD_{r-k}\circ (\Delta_{n-r})_M\\
			&=f_M\Big((n+1)TD_{n+1}-\sum_{k=0}^{n-1}Td_{k+1}\circ TD_{n-k}\Big)+f_M\circ \sum_{l=0}^{n-1}\Big((l+1)TD_{l+1}-\sum_{k=0}^{l}Td_{k+1}\circ TD_{l-k}\Big)\circ (\Delta_{n-l})_M\notag\\
			&=f_M\circ Td_{n+1}.\notag
		\end{align*}
		In a similar way, one obtain, 
		$L_2= f_M\circ (\delta_{k+1})_M.$
		Hence, we obtain $d_{n+1}\circ f_M=f_M\circ (Td_{n+1}+(\delta_{k+1})_M)$.
	\end{proof}
	\begin{cor}
	    Let  $\mathscr C$ be a $\mathfrak F$-linear category and $(T, \theta, \zeta,\Delta)$ be a higher differential monad of order $n$ on $\mathscr C$, where $\Delta$ can be determined by the family of ordinary derivation $\delta=\{\delta_i:T\longrightarrow T\}_{i=0}^n$ on $T$. Let $(M,f_M)$ be an object in $EM_T$ that is equipped with a higher $\Delta$-derivation $D$ of order $n$ on $M$. Then there is a sequence $\{ d_i:M\longrightarrow M\}_{i=0}^n$ of ordinary derivations on $M$ with $d_0=Id_M$ such that $$(i+1)D_{i+1}=\sum_{k=0}^i D_{k+1}\circ d_{i-k}$$ for each $0\leq i< n$.
	\end{cor}
	\begin{eg}
	    Let $(T,\theta,\zeta,\Delta)$ be a higher differential monad as defined in Example \ref{2}. If $(M,f_M)\in EM_T$ and $d:M\longrightarrow M$ is an ordinary $\delta$-derivation on $M$, then $D=\{D_n=\frac{d^n}{n!}\}_{n\geq 0}$ is a higher $\Delta$-derivation on $M$.
	\end{eg}
    \begin{eg}
        Consider the higher differential monad $(T,\theta,\zeta,\Delta)$ as defined in Example \ref{3}. From  \cite[Example 3.6]{DA}, we know that for a module $(M,f_M)\in EM_T$,  $d=f_M\circ\mu_M$ is an ordinary $\delta$-derivation. We define $D=\{D_n:M\longrightarrow M\}_{n\geq 0}$ by 
\begin{equation}\label{BRUU} D_n = \begin{cases} 
      Id & \text{for} ~n=0,\\
          d &  \text{for} ~n=1,\\
          d^n& \text{for all } n\geq 2.
       \end{cases}
\end{equation}

        We now show that $D=\{D_n\}_{n\geq 0}$ is a higher $\Delta$-derivation on $M$. 
        For this, it is enough to show the following diagram
        \begin{equation}\label{4}
			\begin{tikzcd}[row sep=4.8em, column sep = 16em]
				TM \arrow{r}{TD_i+TD_{i-1}\circ (\Delta_{1})_{M}+\ldots+TD_1\circ (\Delta_{i-1})_{M}+(\Delta_i)_{M}}\arrow{d}{f_M} &TM \arrow{d}{f_M}\\
				M \arrow{r}{D_i}& M
			\end{tikzcd}
		\end{equation} 
        commutes for all $i\geq 1.$\\
        For $n=0$, the result is trivial and for $n=1$, the result directly follows from \cite[Example 3.6]{DA}. We now assume that the statement is true for all $i< n$. Then for $i=n$ we have 
        \begin{align}
            D_n\circ f_M&=D_1^n\circ f_M\notag\\
                        &=D_1\circ D_{n-1}\circ f_M\tag{$D_i=D_1^i$}\\
                       &=D_1\circ f_M\circ(TD_{n-1}+TD_{n-2}\circ(\Delta_1)_M+\ldots+(\Delta_{n-1})_M) \tag{by induction}\\
                       &=f_M\circ(TD_1+(\Delta_1)_M)\circ(TD_{n-1}+TD_{n-2}\circ(\Delta_1)_M+\ldots+(\Delta_{n-1})_M) \notag\\
                       &=f_M\circ TD_1\circ(TD_{n-1}+TD_{n-2}\circ(\Delta_1)_M+\ldots+(\Delta_{n-1})_M)+f_M\circ (\Delta_1)_M\circ (TD_{n-1}+TD_{n-2}\circ(\Delta_1)_M+\ldots+(\Delta_{n-1})_M)\notag\\
                       &=f_M\circ (TD_{n}+TD_{n-1}\circ(\Delta_1)_M+\ldots+TD_1\circ(\Delta_{n-1})_M)+f_M\circ (\Delta_1)_M\circ (TD_{n-1}+TD_{n-2}\circ(\Delta_1)_M+\ldots+(\Delta_{n-1})_M).\label{44}
        \end{align}
        Note that for all $1\leq k\leq n$ we have
        \begin{align}
            f_M\circ (\Delta_1)_M\circ TD_{n-k}\circ(\Delta_{k-1})_M
            &=f_M\circ\theta_M\circ(\mu_T-T\mu)_M\circ TD_{n-k}\circ(\Delta_{k-1})_M\notag\\
            &=f_M\circ\theta_M\circ \mu_{TM}\circ TD_{n-k}\circ(\Delta_{k-1})_M-f_M\circ Tf_M\circ T\mu_M\circ TD_{n-k}\circ(\Delta_{k-1})_M\tag{$f_M\circ\theta_M=f_M\circ Tf_M$}\\
            &=f_M\circ\theta_M\circ \mu_{TM}\circ(\Delta_{k-1})_M\circ TD_{n-k}-f_M\circ TD_1\circ TD_{n-k}\circ(\Delta_{k-1})_M\tag{$TD_{n-k}\circ(\Delta_{k-1})_M=(\Delta_{k-1})_M\circ TD_{n-k}$}\\
            &=f_M\circ(\Delta_{k})_M\circ TD_{n-k}-f_M\circ TD_{n-k+1}\circ(\Delta_{k-1})_M\tag{$\theta_M\circ \mu_{TM}\circ(\Delta_{k-1})=(\Delta_{k})_M$}\\
            &=f_M\circ TD_{n-k}\circ(\Delta_{k})_M-f_M\circ TD_{n-k+1}\circ(\Delta_{k-1})_M. \label{61}
        \end{align}
        Therefore the second summand of the equation (\ref{44}) gives 
        \begin{align}
            & f_M\circ (\Delta_1)_M\circ (TD_{n-1}+TD_{n-2}\circ(\Delta_1)_M+\ldots+(\Delta_{n-1})_M)\notag\\
            &=f_M\circ (\Delta_1)_M \circ TD_{n-1}+f_M\circ (\Delta_1)_M\circ TD_{n-2}\circ(\Delta_1)_M+\ldots+f_M\circ (\Delta_1)_M\circ(\Delta_{n-1})_M\notag\\
            &=f_M \circ TD_{n-1}\circ (\Delta_1)_M+f_M\circ TD_{n-2}\circ(\Delta_{2})_M-f_M\circ TD_{n-1}\circ (\Delta_1)_M\tag{by \eqref{61}}\\
            &\qquad+f_M\circ TD_{n-3}\circ(\Delta_{3})_M-f_M\circ TD_{n-2}\circ(\Delta_{2})_M+\ldots+f_M\circ (\Delta_{n})_M-f_M\circ TD_{1}\circ(\Delta_{n-1})_M\notag\\
            &=f_M\circ (\Delta_n)_M.\label{49}
            \end{align}
        Finally, on substituting \eqref{49} in \eqref{44} we get 
        $$D_n\circ f_M=f_M\circ (TD_{n}+TD_{n-1}\circ(\Delta_1)_M+\ldots+TD_1\circ(\Delta_{n-1})_M+(\Delta_n)_M).$$
        For example, if $n=3,$ we obtain
        \begin{align}
             D_3\circ f_M&=D_1^3\circ f_M\notag\\
                        &=D_1\circ D_{2}\circ f_M\tag{$D_i=D_1^i$}\\
                       &=D_1\circ f_M\circ(TD_2+TD_1\circ(\Delta_1)_M+\ldots+(\Delta_2)_M) \tag{by induction}\\
                       &=f_M\circ(TD_1+(\Delta_1)_M)\circ(TD_2+TD_1\circ(\Delta_1)_M+(\Delta_2)_M) \notag\\
                       &=f_M\circ TD_1\circ(TD_2+TD_2\circ(\Delta_1)_M+(\Delta_1)_M)+f_M\circ (\Delta_1)_M\circ (TD_2+TD_1\circ(\Delta_1)_M+(\Delta_2)_M)\notag\\
                       &=f_M\circ (TD_{3}+TD_{2}\circ(\Delta_1)_M+TD_1\circ(\Delta_{2})_M)+f_M\circ (\Delta_1)_M\circ (TD_{2}+TD_{1}\circ(\Delta_1)_M+(\Delta_{2})_M).\label{50}
        \end{align}
        and 
        \begin{align}
           & f_M\circ (\Delta_1)_M\circ (TD_{2}+TD_{1}\circ(\Delta_1)_M+(\Delta_{2})_M)\notag\\
            &=f_M\circ (\Delta_1)_M\circ TD_{2}+f_M\circ TD_{1}\circ(\Delta_{2})_M-f_M\circ TD_{2}\circ(\Delta_{1})_M +f_M\circ(\Delta_{3})_M-f_M\circ TD_{1}\circ(\Delta_{2})_M\tag{by\eqref{61}}\\
            &=f_M\circ (\Delta_3)_M.\tag{$(\Delta_1)_M\circ TD_{2}=TD_2\circ (\Delta_1)_M$}
        \end{align}
        Therefore, we have 
        \begin{equation*}
            D_3\circ f_M=f_M\circ (TD_{3}+TD_{2}\circ(\Delta_1)_M+TD_1\circ(\Delta_{2})_M+(\Delta_3)_M).
        \end{equation*}
    \end{eg}
	\section{Higher differential Torsion Theory}\label{section4}
	In this section, we establish the framework for higher differential torsion theory in the Eilenberg-Moore category of monad. Throughout this section,  we assume that $\mathscr C$ is a Grothendieck category and $\{G_i\}_{i\in \Lambda}$ is a family of finitely generated projective generators for $\mathscr C$. We begin by recalling a few preliminary definitions from \cite{LLV} and \cite{BS}.
	
	\smallskip
	\begin{defn}(\cite[\S 1.1]{BS})
		Let $\mathscr D$ be an abelian category. A pair of full subcategories $\tau=(\mathscr{T},\mathscr{F})$ of $\mathscr D$ is said to be a torsion theory if for any $M\in \mathscr{T}$ and $N\in \mathscr{F}$, $\mathscr{D}(M,N)=0$, and 
		for any $M\in \mathscr{D}$, there exist $M'\in \mathscr{T}$ and $M''\in \mathscr{F}$ such that we have the following short exact sequence
		\[
		0\longrightarrow M'\longrightarrow M\longrightarrow M''\longrightarrow 0
		\] in $\mathscr D.$
		Further, a torsion theory $\tau=(\mathscr{T},\mathscr{F})$ is said to be hereditary if the torsion class $\mathscr{T}$ is closed under subobjects. 
	\end{defn} 
	\begin{defn}\label{EEEEEE}(see~\cite[Proposition 3.2  \& Proposition 3.3]{LLV})\label{D4.2} Let $\mathscr{C}$ be a Grothendieck category, and let $\{G_i\}_{i \in \Lambda}$ be a family of finitely generated projective generators for $\mathscr{C}$. A family of Gabriel filters  is a collection $\mathcal{L} = \{\mathcal{L}_{G_i}\}_{i \in \Lambda}$ that satisfies the following conditions:
		
		\smallskip
		i) For each $i \in \Lambda$, $\mathcal{L}_{G_i}$ consists of subobjects of $G_i$. In particular, 
		$\mathcal{L}_{G_i} \neq \varnothing$ as $G_i\in \mathcal L_{G_i}$.
		
		\smallskip
		ii) If $I \subseteq J \subseteq G_i$ and $I \in \mathcal{L}_{G_i}$, then $J \in \mathcal{L}_{G_i}$.
		
		\smallskip
		iii) For any morphism $g :G_i\longrightarrow G_j$ in $\mathscr C$ and for all $I \in \mathcal{L}_{G_j}$, we have $g^{-1}(I) \in \mathcal{L}_{G_i}$.
		
		\smallskip
		iv) If $I \subseteq J \in \mathcal{L}_{G_j}$ such that $g^{-1}(I) \in \mathcal{L}_{G_i}$ for every $g:G_i\longrightarrow J\subseteq G_j$ (in $\mathscr{C}$), then $I \in \mathcal{L}_{G_j}$.
		
	\end{defn}
	For a Grothendieck category $\mathscr C$, we know from (see \cite[\S 3]{LLV},\cite[\S 2.2]{DA}) that there is a one-to-one correspondence between the hereditary torsion theory $\tau=(\mathscr T,\mathscr F)$ on $\mathscr{C}$ and  the Gabriel filters $\mathcal{L} = \{\mathcal{L}_{G_i}\}_{i \in \Lambda}$ in $\mathscr{C}$ given as follows:
	\begin{equation}\label{STST}
		I\in \mathscr L_{G_i} \Leftrightarrow G_i/I\in \mathscr T~.
	\end{equation}
    We now recall the following results from \cite{DA}.
    \begin{thm}\cite[Proposition 2.2]{DA}\label{T2.3}
		Let $\mathscr C$ be a Grothendieck category and $\{G_i\}_{i\in\Lambda}$ be a family of finitely generated projective generators for $\mathscr C$. Let $(T, \theta, \zeta)$ be a monad on $\mathscr C$ that is exact and preserves colimits. Then $EM_T$ is a Grothendieck category and $\{TG_i\}_{i\in\Lambda}$ gives a family of finitely generated projective generators for $EM_T$.
	\end{thm}
	\begin{defn}(see \cite[Definition 3.3]{DA})
		Let $(T, \theta, \zeta,\delta)$ be a monad on $\mathscr C$ such that $T$ is exact and preserves colimits. Then, a family of Gabriel filters $\mathcal L=\{\mathcal L_{TG_{j}}\}_{j\in \Lambda}$ on $EM_T$ is $\delta$-invariant if for each $j\in\Lambda$ and $I\in \mathcal L_{TG_{j}}$, there exists some $J\in\mathcal L_{TG_{j}}$ such that $\delta_{G_j} (J)\subseteq I .$ 
	\end{defn}
    Let $\tau = (\mathscr T, \mathscr F)$ be a hereditary torsion theory on $EM_{T}$ and $M\in EM_T$. We recall from \cite[\S 3]{DA} that the torsion subobject of $M$ is given by
	\[
	M_\tau:=\sum_{N\in  |M_\tau|}N~,
	\]
	where $|M_\tau|=\{N\subseteq M|\ker(g:TG_i\longrightarrow N)\in \mathcal{L}_{TG_i} $, for all $g\in EM_T(TG_i,N)$ and for all $i\in \Lambda$\}. Then, we have the following definition from \cite{DA}.
	\begin{defn}(see \cite[Definition 3.7]{DA})
	   Let $(T, \theta, \zeta,\delta)$ be a differential monad on $\mathscr C$ and $\tau = (\mathscr T, \mathscr F)$ be a hereditary torsion theory on $EM_{T}$. Then $\tau = (\mathscr T, \mathscr F)$ is called  differential if for any object $(M,f_M)\in EM_{T}$ and any $\delta$-derivation $d$ on $M$, $d(M_{\tau})\subseteq M_{\tau}$.	
	\end{defn}
	We generalize these definitions in the context of higher derivation as follows.
	\begin{defn}
		Let $(T, \theta, \zeta,\Delta)$ be a higher differential monad of order $n$ on $\mathscr C$ such that $T$ is exact and preserves colimits. Then, we say that a family of Gabriel filters $\mathcal L=\{\mathcal L_{TG_{j}}\}_{j\in \Lambda}$ on $EM_T$ is $\Delta$-invariant of order $n$ if for each $j\in\Lambda$ and $I\in \mathcal L_{TG_{j}}$, there exists some $J\in\mathcal L_{TG_{j}}$ such that $(\Delta_i)_{G_j} (J)\subseteq I $ for all $i=0,1,\ldots,n.$ 
	\end{defn}

	\begin{defn}
		Let $(T, \theta, \zeta,\Delta)$ be a higher differential monad of order $n$ on $\mathscr C$ such that $T$ is exact and preserves colimits, and let $\tau = (\mathscr T, \mathscr F)$ be a hereditary torsion theory on $EM_{T}$. Then, we say that $\tau = (\mathscr T, \mathscr F)$ is higher differential of order $n$ if for any object $(M,f_M)\in EM_{T}$ and any higher $\Delta$-derivation $D$ of order $n$ on $M$, $D_i(M_{\tau})\subseteq M_{\tau}$ for all $i=0,1,\ldots,n.
		$	
	\end{defn}
    
 We note that the above notions admit a natural extension to the infinite order case. More precisely, for a higher differential monad $(T, \theta, \zeta,\Delta)$ on $\mathscr C$ a family of Gabriel filters $\mathcal L=\{\mathcal L_{TG_{j}}\}_{j\in \Lambda}$ on $EM_T$ is $\Delta$-invariant if for each $j\in\Lambda$ and $I\in \mathcal L_{TG_{j}}$, there exists some $J\in\mathcal L_{TG_{j}}$ such that $(\Delta_i)_{G_j} (J)\subseteq I $ for all $i\geq0.$ Similarly, a hereditary torsion theory $\tau = (\mathscr T, \mathscr F)$ on $EM_T$ is higher differential if for any object $(M,f_M)\in EM_{T}$ and any higher $\Delta$-derivation $D$ on $M$, $D_i(M_{\tau})\subseteq M_{\tau}$ for all $i\geq 0.
		$
	\begin{Thm}\label{T3.4}
		Let $(T, \theta, \zeta,\Delta)$ be a higher differential monad of order $n$ on $\mathscr C$ that is exact and preserves colimits, and $\tau$ be a hereditary torsion theory on $EM_{T}$. Then the following statements are equivalent.
		
		\smallskip
		i) Every Gabriel filter corresponding to $\tau$ is $\Delta$-invariant of order $n$.
		
		\smallskip
		ii) $\tau$ is a higher differential torsion theory of order $n$.
		
		\smallskip
		iii) For any $(M,f_M) \in EM_T$, every higher $\Delta$-derivation $D$ of order $n$, every submodule $N \subseteq M_\tau$, and every $j \in \{0,\ldots,n\}$, there exists $K \in \mathcal{L}_{TG_i}$ such that
		$(\Delta_j)_{G_i}(K) \subseteq \ker(TG_i \to N).$
		
		\smallskip
		iv) For any $(M,f_M) \in EM_T$, every higher $\Delta$-derivation $D$ of order $n$ on $M$, and every $N \subseteq M_\tau$, there exists $K \in \mathcal{L}_{TG_i}$ such that 
		$(\Delta_j)_{G_i}(K) \subseteq \ker(TG_i \to N)$ for all $j = 0,1,\ldots,n$.
	\end{Thm}
	\begin{proof}
		$i)\implies ii)$. Let $\mathcal{L}=\{\mathcal{L}_{TG_i}\}$ be the Gabriel filter corresponding to $\tau$ and $(M,f_M)\in EM_T$. Let $D:M\longrightarrow M$ be a higher $\Delta$-derivation of order $n$ on $M$ and $N\in |M_\tau|$. To show that $\tau$ is a higher differential torsion theory of order $n$, we will proceed by induction on $n$. For $n=0$, the proof is clear. We assume that the statement is true for $p=0,1,\ldots,n-1$ and we prove it for $p=n$.

        \smallskip
        We consider a morphism $g:G_i\longrightarrow D_nN$ in $EM_T$ and we show that $\textup{ker}(g)\in \mathcal{L}_{TG_i}$. Corresponding to $g$ there is a morphism $\tilde{g}:TG_i\xrightarrow{\zeta_{G_i}}TG_i\xrightarrow{g}D_nN$ in  $\mathscr{C}$. Let $\tilde{D}_n:N\longrightarrow D_nN$ be an epimorphism in $\mathscr{C}$ and $i_n:D_nN\longrightarrow M$ be a monomorphism in $EM_T$. Then, we have $i_n\circ \tilde{D}_n=D_n\circ i$. Since $G_i$ is projective, there exists a morphism $\tilde{h}:G_i\longrightarrow N$ in $\mathscr{C}$ such that $\tilde{D}_n\circ \tilde{h}=\tilde{g}=g\circ \zeta_{G_i}$. Let $h:TG_i\xrightarrow{T\tilde{h}}TN\xrightarrow{f_N}N$ be the morphism corresponding to $\tilde{h}$ in $EM_T$. We set $$H=\cap \{\textup{ker}(TG_i\longrightarrow D_pN)\}_{p=0}^{n-1}\subseteq \ker(h).$$ Since for all $p=1,2\ldots,n-1$, $D_pN\in |M_\tau|$, it is clear that $H\in \mathcal{L}_{TG_i}$. By (i), there exists an object $J\in \mathcal{L}_{TG_i}$ such that $(\Delta_t)_{G_i}(J)\subseteq H $ for all $t=0 ,\ldots ,n$. Let $I=H\cap J$. From the following commutative diagram
		\begin{equation*}
			\begin{tikzcd}[column sep={between origins, 12em}, row sep=large]
				TG_i 
				\arrow[r, "{Ti \circ T\tilde{h}}"]  & TM 
				\arrow[r, "{TD_n + TD_{n-1}\circ  (\Delta_{1})_M + \cdots + (\Delta_{n})_{M}}"] 
				\arrow[d, "f_M"']  & TM  \arrow[d, "f_M"] \\
				& M \arrow[r, "D_n"] & M ,
			\end{tikzcd}
		\end{equation*}
		we have
		\begin{equation}\label{T}
			D_n\circ f_M\circ Ti\circ T\tilde{h}(I)=f_M\circ TD_n\circ Ti\circ T\tilde{h}(I)+\sum_{k=1}^{n-1}f_M\circ TD_k\circ(\Delta_{n-k})_M\circ Ti\circ T\tilde{h}(I)\\ 
			+f_M\circ(\Delta_n)_M\circ Ti\circ T\tilde{h}(I)~.
		\end{equation}
		We note that
			\begin{align*}
				D_n\circ f_M\circ Ti\circ T\tilde{h}(I)&=D_n\circ i\circ f_N\circ T\tilde{h}(I)\tag{ $f_M\circ Ti=i\circ f_N$}\\
				&=D_n\circ i\circ h (I)\tag{ $f_N\circ T\tilde{h}=h$}\\
				&=0 \tag{\text{since $I\subseteq\textup{ker}(h)$}}~.
			\end{align*}
		Also, \begin{align*}
				f_M\circ(\Delta_n)_M\circ Ti\circ T\tilde{h}(I)
				&=f_M\circ Ti \circ(\Delta_n)_N\circ T\tilde{h}(I)\tag{ $(\Delta_i)_M\circ Ti=Ti \circ (\Delta_i)_N$}\\
				&=i\circ f_N \circ T\tilde{h}((\Delta_n)_{G_i}(I))\tag{$(\Delta_n)_{N}\circ T\tilde{h}=T\tilde{h}\circ (\Delta_n)_{G_i}$} \\
				&=i\circ h((\Delta_n)_{G_i}(I))\tag{$f_N\circ T\tilde{h}=h$}\\
				&=0 \tag{$(\Delta_n)_{G_i}(I)\subset H\subseteq \textup{ker}(h)$}~.
			\end{align*}
		Furthermore, as $f_M\circ TD_k\circ(\Delta_{n-k})_M\circ Ti\circ T\tilde{h}:TG_i\longrightarrow D_kN$ is a morphism in $EM_T $ and $(\Delta_{n-k})_{G_i}(I)\subseteq H$, we observe that for any $0< k< n$, \begin{align*}
				f_M\circ TD_k\circ(\Delta_{n-k})_M\circ Ti\circ T\tilde{h}(I)&=f_M\circ TD_k\circ Ti\circ T\tilde{h}((\Delta_{n-k})_{G_i}(I))=0~. 
			\end{align*}
            
        Finally from (\ref{T}), we get 
			\begin{align*}
				0&=f_M\circ TD_n\circ Ti\circ T\tilde{h}(I)\\
				&=f_M\circ Ti_n \circ T\tilde{D_n}\circ  T\tilde{h}(I)\tag{$i_n\circ \tilde{D}_n=D_n\circ i$}\\
				&=f_M\circ Ti_n\circ Tg\circ T\zeta_{G_i}(I)\tag{$\tilde{D}_n\circ \tilde{h}=g\circ \zeta_{G_i}$}\\
				&=i_n\circ g\circ \theta_{G_i}(T\zeta_{G_i}(I))\tag{$f_M\circ Ti_n\circ Tg=i_n\circ g\circ \theta_{G_i}$}\\
				&=i_n\circ g(I)\tag{$\theta_{G_i}\circ T\zeta_{G_i}=1_{TG_i}$}~.
			\end{align*}
		Since $i_n$ is monomorphism, $g(I)=0$. Hence $I\subseteq \text{ker}(g)$. Further, as $I\in \mathcal{L}_{TG_i}$, by property (ii) of Definition \ref{EEEEEE}, we get $\text{ker}(g)\in \mathcal{L}_{TG_i}$. This proves that $D_nN\in |M_\tau|$.
			
		\smallskip
		$ii)\implies iii).$ Let $N\in |M_\tau|$ and  $g:TG_i\longrightarrow N$ be a morphism in $EM_T$. Let $\hat{g}:G_i\longrightarrow N$ be the corresponding morphism given by $\hat{g}=g\circ \zeta_{G_i}$ in $\mathscr{C}$. We now proceed by induction on $n$. For $n=0$, we set $K = \ker(g)\in\mathcal{L}_{TG_i}$. Then $(\Delta_0)_{G_i}(K) = K \subseteq \ker(g)$. Therefore, the statement is true for $n=0$. Now we assume that the statement is true for all $t< n$ and prove it for $t=n$. 
            
        \smallskip
	    Since $\tau$ is higher differential torsion theory of order $n$, from the proof of $(i)\implies(ii)$, we have $D_jN\in |M_\tau|$ for all $j=0,1,\ldots,n$. Therefore, by induction, there exists an object $K_{tj}\in \mathcal{L}_{TG_i}$ such that $(\Delta_t)_{G_i}(K_{tj})\subseteq$ker$(TG_i\longrightarrow D_jN)$ for all $t< n$ and for all $j\leq n$. We let 
        \begin{equation}\label{EqK}
				K=\bigcap_{\substack{t< n \\ j\leq n}}K_{tj} \cap \ker(TG_i\longrightarrow D_{n}N).
        \end{equation} 
        Since $D_nN\in |M_\tau|$,  $K\in \mathcal{L}_{TG_i}$. We now show that $(\Delta_n)_{G_i}(K)\subseteq \ker(g)$. For this, we consider the following commutative diagram 
			\begin{equation*}
				\begin{tikzcd}[column sep={between origins, 12em}, row sep=large]
					TG_i 
					\arrow[r, "{Ti \circ T\hat{g}}"]  & TM 
					\arrow[r, "{TD_n + TD_{n-1}\circ (\Delta_1)_M + \cdots + (\Delta_n)_{M}}"] 
					\arrow[d, "f_M"']  & TM  \arrow[d, "f_M"] \\
					& M \arrow[r, "D_n"] & M~,
				\end{tikzcd}
			\end{equation*}
			i.e, we have 
			\begin{equation}
				D_n\circ f_M\circ Ti\circ T\hat{g}(K)=f_M\circ TD_n\circ Ti\circ T\hat{g}(K)+\sum_{k=1}^{n-1}f_M\circ TD_k\circ(\Delta_{n-k})_M\circ Ti\circ T\hat{g}(K)\\ 
				+f_M\circ(\Delta_n)_M\circ Ti\circ T\hat{g}(K).
			\end{equation}
			We note that 
            \begin{align*}
				D_n\circ f_M\circ Ti\circ T\hat{g}(K)
				&=D_n\circ i \circ f_N\circ T\hat{g}(K)\tag{$f_M\circ Ti=i\circ f_N$}\\
				&=D_n\circ i\circ g(K)\tag{ $f_N\circ T\hat{g}=g$}\\
				&=0\tag{ $K\subseteq \ker(g)$}~.
			\end{align*}
			Furthermore, for $0<k\leq n$, we get
			\begin{align*}
				f_M\circ TD_k\circ(\Delta_{n-k})_M\circ Ti\circ T\hat{g}(K)
                &=f_M\circ TD_k\circ Ti\circ(\Delta_{n-k})_N \circ T\hat{g}(K)\tag{$(\Delta_{n-k})_M\circ Ti=Ti\circ(\Delta_{n-k})_N $}\\
				&=f_M\circ TD_k\circ Ti\circ T\hat{g}((\Delta_{n-k})_{G_i}(K))\tag{$(\Delta_{n-k})_N \circ T\hat{g}=T\hat{g}\circ(\Delta_{n-k})_{G_i}$}\\
				&=f_M\circ T(D_k\circ i\circ \hat{g})((\Delta_{n-k})_{G_i}(K))\\
				&=(\widetilde{D_k\circ i\circ \hat{g}})((\Delta_{n-k})_{G_i}(K))\\
				&=0,
			\end{align*}
			where the last equality follows from the fact that $D_kN\in |M_\tau|$ and $(\widetilde{D_k\circ i\circ \hat{g}}):TG_i\longrightarrow D_kN$ is a morphism in $EM_T $. Hence, we obtain 
			\begin{align*}
				0=f_M\circ(\Delta_n)_M\circ Ti\circ T\hat{g}(K)
				&=f_M\circ Ti\circ T\hat{g}((\Delta_n)_{G_i}(K))\tag{$(\Delta_n)_M\circ Ti\circ T\hat{g}=Ti\circ T\hat{g}\circ(\Delta_n)_{G_i}$}\\
				&=i\circ f_N\circ T\hat{g}((\Delta_n)_{G_i}(K))\tag{$f_M\circ Ti=i\circ f_N$}\\
				&=i \circ g((\Delta_n)_{G_i}(K))\tag{$f_N\circ T\hat{g}=g$}.
			\end{align*}
			Since $i$ is a monomorphism, we get $ g((\Delta_n)_{G_i}(K))=0$. This shows that $(\Delta_n)_{G_i}(K)\subseteq \ker(g)$.
			
			\smallskip 
			$iii)\implies iv).$ Let $N\in |M_\tau|$. Then by $(iii)$, for each $j=0,1,\ldots,n$, there exists $K_j\in \mathcal{L}_{TG_i}$ such that $(\Delta_j)_{G_i}(K_j)\subseteq \ker(TG_i\longrightarrow N).$ The proof now follows by taking $K=\bigcap_{0\leq j\leq n} K_j $. 
			
			\smallskip
			$iv)\implies i).$ Given $I\in \mathcal{L}_{TG_i}$, we know from (\ref{STST}) that $TG_i/I\in \mathcal{T}$. Now, by (iv), there exists an object $K\in \mathcal{L}_{TG_i}$ such that $(\Delta_j)_{G_i}(K)\subseteq \ker (TG_i\longrightarrow TG_i/I)=I$ for each $j=0,1,\ldots,n.$ This completes the proof.
		\end{proof}
		\begin{rem}
		Note that Theorem \ref{T3.4} continues to hold in the setting of a higher differential
monad $(T,\theta,\zeta,\Delta)$ on $\mathscr{C}$, where $T$ is exact and
preserves colimits and $\tau$ is a hereditary torsion theory on $EM_T$. The argument proceeds in a manner analogous to the proof given
above.
		\end{rem}

		\section{Higher order derivations on module of quotients}\label{section5}
			We continue with a Grothendieck category $\mathscr C$ having a family $\{G_i\}_{i\in\Lambda}$ of finitely generated projective generators for $\mathscr C$. Let  \(\tau = (\mathscr{T}, \mathscr{F})\) be a hereditary torsion theory on \(EM_T\) and \(\mathcal{L} = \{\mathcal{L}_{TG_i}\}_{G_i \in \mathscr{C}}\) be the family of Gabriel filter associated with \(\tau\). Then, we know from \cite{GG} that for every object \((M, f_M) \in EM_T\), there exists a functor
			\[
			Q_\tau : EM_T \longrightarrow EM_T, \qquad M \longmapsto Q_\tau(M),
			\]
			defined by setting
			\[
			Q_\tau(M)(TG_i) = \varinjlim_{I \in \mathcal{L}_{TG_i}} EM_T \bigl(I, M / M_\tau \bigr).
			\]
			The object \(Q_\tau(M)\) is referred to as the module of quotients of \(M\). Let $p:M\longrightarrow  M/M_{\tau}$ denotes the canonical epimorphism in $EM_{T}$. For any morphism \(f : TG_i \to M\) in \(EM_T\), we know from \cite[\S\,2.2]{GG} that there exists a morphism
\[
\Phi_M : M \longrightarrow Q_\tau(M)
\]
in \(EM_T\) defined by setting
\[
\Phi_M(TG_i)(f) = \eta_i\bigl(p \circ f\bigr),
\]
where
\[
\eta_i : EM_T(TG_i,\, M / M_\tau) \longrightarrow 
\varinjlim_{I \in \mathcal{L}_{TG_i}} EM_T\!\bigl(I,\, M / M_\tau\bigr)
\]
is the canonical morphism. Further, it is clear from \cite[\S 2.2]{GG} that $\Phi$ is functorial in $M$, i.e., for any morphism \(g : TG_j \longrightarrow TG_i\) in \(EM_T\), there is an induced morphism
			\[
			Q_\tau(M)(g) : Q_\tau(M)(TG_i) \longrightarrow Q_\tau(M)(TG_j),
			\]
			such that the following equality holds:
			\begin{equation}
				Q_{\tau}(M)(g)\circ \Phi_{M}(TG_i)(g)=\Phi_{M}(TG_j)(g)\circ EM_T(-,M)(g).
			\end{equation}
			From \cite{GG}, we also know that $\ker(\Phi_M)$ and $\operatorname{coker}(\Phi_M)$ are $\tau$-torsion in $EM_T$. Furthermore, since $\Phi_M$ is functorial and $Q_{\tau}(M)=Q_{\tau}(M/M_\tau)$ we have the following commutative diagram 
            \begin{equation*}
                \begin{tikzcd}[row sep=3.8em, column sep = 7em]
                  M \arrow[r, "p"] \arrow[d, "\Phi_M"'] 
                  & M/M_\tau \arrow[d, "\Phi_{M/M_\tau}"] \\
                 Q_\tau(M) \arrow[r, "\cong"] 
                  & Q_\tau(M/M_\tau),
                \end{tikzcd}
            \end{equation*}
            
          i.e,$\Phi_M(M_\tau)=\Phi_{M/M_\tau}\circ p(M_\tau)=0$, which implies $M_\tau \subseteq\operatorname{ker}(\Phi_M)$. Therefore, $\operatorname{ker}(\Phi_M)=M_\tau.$
          
			\smallskip
			In this section, we consider a higher differential monad \((T, \theta, \zeta, \Delta)\) of order $n$ on $\mathscr C$ where $T$ is exact and preserve colimits. Given a higher $\Delta$-derivation of order $n$ on an object $M \in EM_T$, we now extend it to the module of quotients $Q_{\tau}(M)$. From now on, we assume that $\tau$ is a higher differential torsion theory of order $n$. Then by Theorem~\ref{T3.4}, the Gabriel filter corresponding to $\tau$ is $\Delta$-invariant of order $n$.
            
            \begin{rem} Note that all the results established in this section continue to hold for higher differential
monad of infinite order. We therefore treat only the
finite order case, as the proofs in the general case are entirely analogous.
            \end{rem}
			
			\begin{lem}\label{M}
				Let $(M,f_M)\in EM_T$ be a torsion-free module equipped with a higher $\Delta$-derivation $D$ of order $n$ and let $f\in Q_{\tau}(M)(TG_i)$ be a morphism represented by $f:I\longrightarrow M$ for some $I\in \mathcal{L}_{TG_i}$. Let  $J\in\mathcal{L}_{TG_i}$ be such that $(\Delta_k)_{G_i}(J)\subseteq I $ for all $k=0,1,\ldots,n$. Set $K = I \cap J$. Define maps $\widetilde{D}_k f : K \longrightarrow M$ by 
                \begin{equation}\label{eq4.1} \widetilde{D}_k f(K) = \begin{cases} 
      f(K) & \text{for} ~k=0,\\
           D_1 f(K) - f\!\left((\Delta_1)_{G_i}(K)\right) &  \text{for} ~k=1,\\
          D_k\!\bigl(f(K)\bigr)
 - f\!\left((\Delta_k)_{G_i}(K)\right)
 - \sum_{\substack{t+j=k \\ t,j>0}}
   \widetilde{D}_t f\!\left((\Delta_j)_{G_i}(K)\right) & \text{for all } 2\leq k\leq n.
       \end{cases}
\end{equation}

Then, for every $0 \le k \le n$, the maps $\widetilde{D}_k f$ are morphisms in $Q_{\tau}(M)(T G_i)$.
			\end{lem}
			
			\begin{proof}
				Since $(K,f_k)\in \mathcal{L}_{TG_i}$ and $\mathcal{L}_{TG_i}$ is $\Delta$-invariant of order $n$, there exists an object $K'\in \mathcal{L}_{TG_i}$ such that $(\Delta_t)_{G_i}(K')\subseteq K$ for all $0\leq t\leq n$. We now proceed by induction on $n$.  Note that for $n=0$, $D_0f=f$,  the result is trivially true. For $n = 1$, since $(\Delta_1)_{G_i}$ is an ordinary derivation on $T$, the result follows directly from \cite[Lemma 4.1]{DA}. We now prove the result for $n=3$. The general case can then be established in an analogous manner.\\
                To prove that $\widetilde{D}_3f$ is a morphism in $Q_\tau (M)(TG_i)$, it is enough to show that the following diagram commutes
				\begin{equation*}
                    \begin{tikzcd}[row sep=3em, column sep=5.5em]
                       TK' \arrow[r, "Ti"] 
                           \arrow[d, "f_{K'}"'] 
                     & TK \arrow[r, "T(\widetilde{D}_3 f)"] 
                           \arrow[d, "f_{K}"'] 
                     & TM \arrow[d, "f_{M}"] \\
                        K' \arrow[r, "i"'] 
                     & K \arrow[r, "\widetilde{D}_3 f"'] 
                     & M,
                    \end{tikzcd}
                \end{equation*} 
				i.e., \begin{equation}\label{eq4.2}
					\widetilde{D}_3f(i(f_{K'}(TK')))=f_M(T(\widetilde{D}_3f)(Ti(TK')))~.
				\end{equation}
				By Definition in \eqref{eq4.1}, we have
				\begin{equation}
					\widetilde{D}_3f(i(f_{K'}(TK')))=D_3(f(i(f_{K'}(TK'))))-f((\Delta_3)_{G_i}(i(f_{K'}(TK')))-\sum_{\substack{t+j=3\\t,j{>}0}}\widetilde{D}_tf((\Delta_j)_{G_i}(i(f_{K'}(TK')))) .\label{eq:A}
				\end{equation}
				We see that the right-hand side of \eqref{eq4.2} gives
				\begin{align}
					f_M(T(\widetilde{D}_3f)(Ti(TK')))=f_M\circ T((\widetilde{D}_3f)(i(K')))
					&=f_M(T(D_3f(i(K'))-f((\Delta_3)_{G_i}(i(K')))-\sum_{\substack{t+j=3\\t,j{>}0}}\widetilde{D}_tf(\Delta_j)_{G_i}(i(K')))) \nonumber\\
					&= f_M(T(D_3f)(Ti(TK')))-f_M(Tf(T(\Delta_3)_{G_i})(Ti(TK')))\nonumber\\
					&\qquad-f_M\sum_{\substack{t+j=3\\t,j{>}0}}T(\widetilde{D}_tf)T(\Delta_j)_{G_i}(Ti(TK')). \label{eq:B}
				\end{align}
				
				We now consider the following commutative diagram 
				\begin{equation}
					\begin{tikzcd}[column sep={1cm,1cm,2cm},row sep=large] 
						TK' \arrow[r, "Ti"] \arrow[ "f_{K'}",d] & 
						TK \arrow["Tf",r] \arrow[ "f_K",d] & 
						TM \arrow[r, "{TD_3 + TD_2\circ(\Delta_1)_M + TD_1\circ(\Delta_2)_M + (\Delta_3)_M}"] \arrow[ "f_M",d] & 
						TM \arrow[ "f_M",d] \\
						K' \arrow[r, "i"] & 
						K \arrow[r, "f"] & 
						M \arrow[r, "D_3"] & 
						M,
					\end{tikzcd}
				\end{equation} 
                $i.e,$
                \begin{equation}\label{eq:D}
                    D_3(f(i(f_{k'}(TK'))))=f_M\circ (TD_3+ TD_2\circ(\Delta_1)_M + TD_1\circ(\Delta_2)_M + (\Delta_3)_M)\circ Tf\circ Ti(TK') .
                \end{equation}
            On substituting the value of $f_M(T(D_3(Tf(Ti(TK')))))$ in \eqref{eq:B} we obtain
            \begin{align}\label{ZZZ}
                f_M(T(\widetilde{D}_3f)(Ti(TK')))&=D_3(f(i(f_{k'}(TK'))))-f_M\circ ( TD_2\circ(\Delta_1)_M + TD_1\circ(\Delta_2)_M + (\Delta_3)_M)\circ Tf\circ Ti(TK')\notag\\
                \qquad&-f_M(Tf(T(\Delta_3)_{G_i})(Ti(TK')))-f_M\sum_{\substack{t+j=3\\t,j{>}0}}T(\widetilde{D}_tf)T(\Delta_j)_{G_i}(Ti(TK')).
            \end{align}
           From (\ref{eq:A}) and (\ref{ZZZ}), it suffices to verify that the following equation holds:
				\begin{align}\label{eq4.7}
					f((\Delta_3)_{G_i}(i(f_{K'}(TK')))+\sum_{\substack{t+j=3\\t,j{>}0}}\widetilde{D}_tf(\Delta_j)_{G_i}(i(f_{K'}(TK'))) &=f_M(TD_2\circ(\Delta_{1_M})+TD_1\circ((\Delta_2)_M)+(\Delta_3)_M)(Tf(Ti(TK')))\\\notag
					&\qquad +f_M(Tf(T(\Delta_3)_{G_i}(Ti(TK'))))+f_M\Big(\sum_{\substack{t+j=3\\t,j{>}0}}T(\widetilde{D}_tf)\circ T(\Delta_j)_{G_i}(Ti(TK'))\Big).
				\end{align}
				Since for all $t=0,1,\ldots,n,$ $(\Delta_t)_{G_i}(K')\subseteq K$ and $f_K$ , $(\Delta_t)_K$ are restrictions of $\theta_{G_i}$ and $(\Delta_t)_{TG_i}$ respectively, we get 
                \begin{align}
					(\Delta_k)_{G_i}(i(f_{K'}(TK')))&=(\Delta_k)_{G_i}(f_k(Ti(TK')))\notag\\
					&= (\Delta_k)_{G_i}(\theta_{G_i}(Ti(TK')))\notag\\
					&=\theta_{G_i}((\Delta_k)_{TG_i}+\ldots+T(\Delta_k)_{G_i})(Ti(TK'))\tag{by\eqref{eqE}}\\
					&=f_K((\Delta_k)_{TG_i}+\ldots+T(\Delta_k)_{G_i})(Ti(TK')).\label{A1}
				\end{align}
				Then, on the right-hand side (\textbf{RHS}) of \eqref{eq4.7}, we see that
				\begin{align}\label{RHS}
					&f_M\circ\left(TD_2\circ(\Delta_1)_M + TD_1\circ(\Delta_2)_M + (\Delta_3)_M\right)\circ\left(Tf(Ti(TK'))\right) + f_M\circ Tf (T(\Delta_3)_{G_i}(Ti(TK'))) + f_M(\sum_{\substack{t+j=3\\ t,j>0}} T(\widetilde{D}_tf)T(\Delta_j)_{G_i}(Ti(TK')))\notag \\
					&= f_M\circ\left(TD_2\circ(\Delta_1)_M + TD_1\circ(\Delta_2)_M\right)\circ\left(Tf(Ti(TK'))\right) + f_M\circ(\Delta_3)_M\circ(Tf(Ti(TK')))  + f_M\circ Tf\circ(T(\Delta_3)_{G_i}\circ(Ti(TK'))  \notag\\&\qquad \qquad +\sum_{\substack{t+j=3\\ t,j>0}} (\widetilde{D}_tf)\circ f_k \circ T(\Delta_j)_{G_i}\circ(Ti(TK')),
				\end{align}
                where the last equality holds from the fact that $(\widetilde{D}_tf)\circ f_k=f_M\circ T(\widetilde{D}_tf)$ for $t=1,2$.

                \smallskip
				On the other hand, the left-hand side (\textbf{LHS}) of \eqref{eq4.7} is
				\begin{align}\label{LHS}
					&f((\Delta_3)_{G_i}(i(f_{K'}(TK')))+\sum_{\substack{t+j=3\\t,j{>}0}}\widetilde{D}_tf(\Delta_j)_{G_i}(i(f_{K'}(TK')))\\
					&=f\circ f_K((\Delta_3)_{TG_i}+(\Delta_2)_{TG_i}\circ T(\Delta_1)_{G_i}+(\Delta_1)_{TG_i}\circ T(\Delta_2)_{G_i}+T(\Delta_3)_{G_i})\circ (Ti(TK'))\notag\\
					&\qquad+\sum_{\substack{t+j=3\\t,j{>}0}}(\widetilde{D}_tf)\circ f_K\circ((\Delta_j)_{TG_i}+(\Delta_{j-1})_{TG_i}\circ T(\Delta_1)_{G_i}+\ldots+T(\Delta_j)_{G_i})\circ(Ti(Tk')))\tag{by \eqref{A1}} \\
					&=f\circ f_K((\Delta_3)_{TG_i}+(\Delta_2)_{TG_i}\circ T(\Delta_1)_{G_i}+(\Delta_1)_{TG_i}\circ T(\Delta_2)_{G_i}+T(\Delta_3)_{G_i})\circ (Ti(TK'))\notag\\
					&\qquad+\sum_{\substack{t+j=3\\t,j{>}0}}(\widetilde{D}_tf)\circ f_K((\Delta_j)_{TG_i}+(\Delta_{j-1})_{TG_i}\circ T(\Delta_1)_{G_i}+...+(\Delta_{1})_{TG_i}\circ T(\Delta_{j-1})_{G_i})\circ (Ti(TK')))+\sum_{\substack{t+j=3\\t,j{>}0}}(\widetilde{D}_tf)\circ f_K\circ T(\Delta_{j})_{G_i}\circ (Ti(TK'))\notag \\
					&=f\circ f_K((\Delta_3)_{TG_i}+(\Delta_2)_{TG_i}\circ T(\Delta_1)_{G_i}+(\Delta_1)_{TG_i}\circ T(\Delta_2)_{G_i}+T(\Delta_3)_{G_i})\circ (Ti(TK'))\notag\\
					&\qquad+\sum_{\substack{t+j=3\\t,j{>}0}}f_M \circ T(\widetilde{D}_tf)\circ ((\Delta_j)_{TG_i}+(\Delta_{j-1})_{TG_i}\circ T(\Delta_1)_{G_i}+\ldots+(\Delta_1)_{TG_i}\circ T(\Delta_{j-1})_{G_i})\circ (Ti(Tk')))\tag{{by induction}}\\
					&\qquad+\sum_{\substack{t+j=3\\t,j{>}0}}(\widetilde{D}_tf)\circ f_K\circ (T(\Delta_j)_{G_i})\circ (Ti(TK'))\notag
                    \end{align}
				\begin{align*}
					&=f_M\circ Tf\circ ((\Delta_3)_{TG_i}+(\Delta_2)_{TG_i}\circ T(\Delta_1)_{G_i}+(\Delta_1)_{TG_i}\circ T(\Delta_2)_{G_i}+T(\Delta_3)_{G_i})\circ (Ti(TK'))\tag{$f\circ f_K=f_M\circ Tf$}\\
					&\qquad +f_M\circ T(\widetilde{D}_1f)\circ ((\Delta_2)_{TG_i}+(\Delta_1)_{TG_i}\circ T(\Delta_1)_{G_i})\circ (Ti(TK'))+f_M\circ T(\widetilde{D}_2f)\circ (\Delta_1)_{TG_i}\circ (Ti(TK'))\notag\\
					&\qquad+\sum_{\substack{t+j=3\\t,j{>}0}}(\widetilde{D}_tf)\circ f_K\circ (T\Delta_{j})_{G_i}\circ (Ti(TK'))\notag\\
					&=f_M\circ Tf\circ ((\Delta_3)_{TG_i}+(\Delta_2)_{TG_i}\circ T(\Delta_1)_{G_i}+(\Delta_1)_{TG_i}\circ T(\Delta_2)_{G_i}+T(\Delta_3)_{G_i})\circ (Ti(TK'))\notag\\
					&\qquad +f_M\circ T\Big(D_1f-f(\Delta_1)_{G_i}\Big)\circ ((\Delta_2)_{TG_i}+(\Delta_1)_{TG_i}\circ T(\Delta_1)_{G_i})\circ (Ti(TK'))+f_M\circ T\Big(D_2f-f(\Delta_2)_{G_i}-\widetilde{D_1}f(\Delta_{1})_{G_i}\Big)\circ (\Delta_1)_{TG_i})\notag\\&\qquad\circ (Ti(TK')) +\sum_{\substack{t+j=3\\t,j{>}0}}(\widetilde{D}_tf)\circ f_K\circ T(\Delta_{j})_{G_i}\circ (Ti(TK'))\notag\\
					&=f_M\circ Tf\circ ((\Delta_3)_{TG_i}+(\Delta_2)_{TG_i}\circ T(\Delta_1)_{G_i}+(\Delta_1)_{TG_i}\circ T(\Delta_2)_{G_i}+T(\Delta_3)_{G_i})\circ (Ti(TK'))\notag\\
					&\qquad+ f_M\circ TD_1\circ Tf\circ ((\Delta_2)_{TG_i}+(\Delta_1)_{TG_i}\circ T(\Delta_1)_{G_i})\circ (Ti(TK'))-f_M \circ Tf\circ T(\Delta_1)_{G_i}\circ ((\Delta_2)_{TG_i}\notag\\
					&\qquad+(\Delta_1)_{TG_i}\circ T(\Delta_1)_{G_i})\circ (Ti(TK'))+f_M\circ TD_2\circ Tf\circ (\Delta_1)_{TG_i}\circ (Ti(TK'))-f_M\circ Tf\circ T(\Delta_2)_{G_i}\circ (\Delta_1)_{TG_i}\circ (Ti(TK'))\notag\\
					&\qquad-f_M\circ T(\widetilde{D_1}f)\circ T(\Delta_{1})_{G_i}\circ (\Delta_1)_{TG_i}\circ (Ti(TK'))+\sum_{\substack{t+j=3\\t,j{>}0}}(\widetilde{D}_tf)\circ f_K\circ T(\Delta_j)_{G_i}\circ (Ti(TK'))\notag\\               &=\textbf{RHS}+f_M\circ Tf\circ ((\Delta_2)_{TG_i}\circ T(\Delta_1)_{G_i}+(\Delta_1)_{TG_i}\circ T(\Delta_2)_{G_i})\circ (Ti(TK'))+f_M\circ TD_1\circ Tf\circ(\Delta_1)_{TG_i}\circ T(\Delta_1)_{G_i}\circ (Ti(TK'))\notag\\
					&\qquad-f_M \circ Tf\circ T(\Delta_1)_{G_i}\circ ((\Delta_2)_{TG_i}+(\Delta_1)_{TG_i}\circ T(\Delta_1)_{G_i})\circ (Ti(TK'))-f_M\circ Tf\circ T(\Delta_2)_{G_i}\circ (\Delta_1)_{TG_i}\circ (Ti(TK'))\notag\\
					&\qquad- f_M\circ TD_1\circ Tf\circ T(\Delta_{1})_{G_i}\circ (\Delta_1)_{TG_i}\circ (Ti(TK'))+f_M\circ Tf\circ T(\Delta_1)_{G_i}\circ T(\Delta_{1})_{G_i}\circ (\Delta_1)_{TG_i}\circ (Ti(TK'))\tag{by \eqref{RHS}} \\
                    &=\textbf{RHS}+0.\tag{as~ $(\Delta_l)_{TG_i}\circ T(\Delta_m)_{G_i}=T(\Delta_m)_{G_i}\circ (\Delta_l)_{TG_i}$}
				\end{align*}
				This completes the proof. 
			\end{proof}
            \smallskip
           Therefore, by Lemma \ref{M}, there exists a family of morphisms $\widetilde{D}:=\{\widetilde{D}_k:Q_\tau (M)\longrightarrow Q_\tau (M)\}_{k=1}^n$ in $\mathscr{C}$ defined as follows:
            \begin{equation}\label{SD}
            \widetilde{D}_k(TG_i):Q_\tau (M)(TG_i)\longrightarrow Q_\tau (M)(TG_i),\hspace{.5cm}f\mapsto \widetilde{D}_kf 
            \end{equation}
            for $M\in EM_T $ and $f\in Q_\tau (M)(TG_i).$
			\begin{lem}\label{K2}
				Let $(M,f_M)\in EM_T$ be a torsion free module and $D$ be a $\Delta$-derivation of order $n$ on $M.$ If $\widetilde{D}:=\{\widetilde{D_k}:Q_\tau (M)\longrightarrow Q_\tau (M)\}_{k=0}^n$ is the family of morphism as defined in \eqref{SD}, then the following diagram commutes 
				\begin{equation}\label{eq3.5*}
					\begin{tikzcd}[row sep=3.8em, column sep = 8em]
						M \arrow{r}{\Phi_M }\arrow{d}{D_n} & Q_{\tau}(M) \arrow{d}{\widetilde{D_n}}\\
						M \arrow{r}{\Phi_M}& Q_{\tau}(M)
					\end{tikzcd}
				\end{equation}
                for all $k=0,1,\ldots,n.$
			\end{lem}
			\begin{proof}
            We consider a morphism $f: TG_i\longrightarrow M$ in $EM_T$. Then, there exists a corresponding morphism $\tilde{f}:G_i\longrightarrow M$ in $\mathscr{C}$ given by $\tilde{f}:G_i\xrightarrow{\zeta_{G_i}}TG_i\xrightarrow{f}M$. Further, $D_k\circ\tilde{f}:G_i\xrightarrow{\tilde{f}} M\xrightarrow{D_k} M$ is a morphism in $\mathscr{C}$, for which there exists a corresponding morphism $\hat{D}_k(f)\in EM_T(TG_i,M)$ given by $\hat{D}_k(f):TG_i\xrightarrow{T\zeta_{G_i}}TTG_i\xrightarrow{Tf}TM\xrightarrow{TD_k}TM\xrightarrow{f_M}M $.
			Note that to prove the commutativity of diagram \eqref{eq3.5*}, it is enough to show that the following diagram commutes 
				\begin{equation}\label{eqC}
					\begin{tikzcd}[row sep=3.8em, column sep = 8em]
						EM_T(TG_i,M)\arrow{r}{\Phi_M (TG_i)}\arrow{d}{\hat{D}_k(TG_i)} & Q_{\tau}(M)(TG_i) \arrow{d}{{\widetilde{D}_k(TG_i)}}\\
						EM_T(TG_i,M) \arrow{r}{\Phi_M(TG_i)}& Q_{\tau}(M)(TG_i)
					\end{tikzcd}
				\end{equation}  
                for all $k = 0,1,\ldots,n$.\\
				We proceed by induction on $n$. For $k = 0$, the diagram \eqref{eqC} commutes trivially, and for $k = 1$, its commutativity follows from \cite[Lemma~4.1]{DA}. We now assume that the diagram \eqref{eqC} commutes for all  $k < n$. Note that
				\begin{align}
					\widetilde{D}_n(f)(TG_i)&=D_n(f(TG_i))-f((\Delta_n)_{Gi}(TG_i))-\sum_{\substack{t+j=n\\t,j{>}0}}\widetilde{D}_tf((\Delta_j)_{G_i}(TG_i))\notag\\
					&=D_n(f(1_{TG_i}( TG_i)))-f((\Delta_n)_{Gi}(TG_i))-\sum_{\substack{t+j=n\\t,j{>}0}}\widetilde{D}_tf((\Delta_j)_{G_i}(TG_i))\notag\\
					&=D_n(f(\theta_{G_i}\circ T\zeta_{G_i}(TG_i)))-f((\Delta_n)_{Gi}(TG_i))-\sum_{\substack{t+j=n\\t,j{>}0}}\widetilde{D}_tf((\Delta_j)_{G_i}(TG_i))\tag{$\theta_{G_i}\circ T\zeta_{G_i}=1_{TG_i}$}\\
					&=D_n(f_M(Tf( T\zeta_{G_i}(TG_i))))-f((\Delta_n)_{Gi}(TG_i))-\sum_{\substack{t+j=n\\t,j{>}0}}\widetilde{D}_tf((\Delta_j)_{G_i}(TG_i))\tag{$f\circ \theta_{G_i}=f_M\circ Tf$}\\
					&=f_M(TD_n+TD_{n-1}\circ (\Delta_1)_M+\dots+(\Delta_n)_M)\circ (Tf( T\zeta_{G_i}(TG_i)))-f((\Delta_n)_{Gi}(TG_i))\tag{by \eqref{eq3.1}}\\
					&\qquad -\sum_{\substack{t+j=n\\t,j{>}0}}\widetilde{D}_tf((\Delta_j)_{G_i}(TG_i)).\label{S}
				\end{align}
				Further, we see that
				\begin{align*}
					f_M((\Delta_n)_M(Tf( T\zeta_{G_i}(TG_i))))&=f_M(Tf((\Delta_n)_{TG_i}(T\zeta_{G_i}(TG_i))))\tag{$(\Delta_n)_M\circ Tf=Tf\circ (\Delta_n)_{TG_i}$}\\&=f_M(Tf(T\zeta_{G_i}(\Delta_n)_{G_i}(TG_i)))\tag{$(\Delta_n)_{TG_i}\circ T\zeta_{G_i}=T\zeta_{G_i}\circ (\Delta_n)_{G_i}$}\\&=f((\Delta_n)_{G_i}(TG_i))\tag{$f_M\circ Tf\circ T\zeta_{G_i}=f$}.
				\end{align*}
				Now, for $0<k<n$, we get
				\begin{align*}
					\widetilde{D}_kf((\Delta_{n-k})_{G_i}(TG_i))&=\hat{D}_kf\circ ((\Delta_{n-k})_{G_i}(TG_i))\tag{by induction}\\&=f_M\circ TD_k\circ Tf \circ T\zeta_{G_i}\circ ((\Delta_{n-k})_{G_i}(TG_i))\tag{$\hat{D}_kf=f_{M}\circ Tf\circ T\zeta_{G_i}$}\\&=f_M\circ TD_k\circ Tf \circ (\Delta_{n-k})_{TG_i}\circ T\zeta_{G_i}\circ (TG_i)\tag{$T\zeta_{G_i}\circ (\Delta_{n-k})_{G_i}=(\Delta_{n-k})_{TG_i}\circ T\zeta_{G_i}$}\\&=f_M\circ TD_k\circ (\Delta_{n-k})_{M} \circ Tf \circ T\zeta_{G_i}\circ (TG_i).
				\end{align*}
				Finally on substituting the above expressions in (\ref{S}), we obtain
				\begin{align*}
					\widetilde{D}_n(f)(TG_i)&=f_M(TD_n(Tf(T\zeta_{G_i}(TG_i))))\\&=\hat{D}_n(f)(TG_i).
				\end{align*}
				Therefore $\widetilde{D}_n(f)(TG_i)=\hat{D}_n(f)(TG_i)$ for all $n$, and consequently the diagram in (\ref{eqC}) commutes for all $0\leq k\leq n$.
			\end{proof}
			\begin{Thm}\label{K}
				Let $(T,\theta,\zeta,\Delta)$ be a higher differential monad of order $n$ on $\mathscr{C}$ that is exact and preserves colimits and let $\tau=(\mathscr{T},\mathscr{F})$ be a hereditary differential torsion theory of order $n$ on $EM_T$. Let $(M,f_M)\in EM_T$ be a torsion-free module equipped with a higher $\Delta$-derivation $D$ of order $n$. Then the family of morphisms $\widetilde{D}=\{\widetilde{D_k}:Q_\tau (M)\longrightarrow Q_\tau (M)\}_{k=0}^n$ where $\widetilde{D}_0=Id_{Q_\tau (M)},$ the identity morphism on $Q_\tau (M)$ in $\mathscr{C}$ is a higher $\Delta$-derivation on $(Q_\tau (M),f_{Q_\tau (M)})$  of order $n$.
			\end{Thm}
			\begin{proof}
				Let $g\in Q_{\tau}(M)(TG_i)$ be represented by the morphism $g:I\longrightarrow M$ for some $(I,g_I)\in \mathcal{L}_{TG_i}$. To prove that $\widetilde{D}$ is a higher derivation on $Q_\tau (M)$, we show that the following diagram commutes in $\mathscr{C}$ for all $k=0,1,\ldots,n$. 
				\begin{equation}\label{eq4.12}
					\begin{tikzcd}[column sep=12em, row sep=large]
						TI \arrow[r, "Tg", shorten >=1pt, shorten <=1pt] \arrow[d, "g_I"'] & 
						TM \arrow[r, "T\Phi_{M}", shorten >=1pt, shorten <=1pt] \arrow[d, "f_M"'] & 
						TQ_{\tau}(M) 
						\arrow[r, 
						"{T\widetilde{D}_k + T\widetilde{D}_{k-1}\circ (\Delta_1)_{Q_{\tau}(M)} + \dots + (\Delta_k)_{Q_\tau(M)}}", 
						shorten >=0pt, 
						shorten <=0pt, 
						column sep=5cm] 
						\arrow[d, "f_{Q_{\tau}(M)}"'] & 
						TQ_{\tau}(M) \arrow[d, "f_{Q_{\tau}(M)}"] \\
						I \arrow[r, "g"] & 
						M \arrow[r, "\Phi_M"] & 
						Q_{\tau}(M) \arrow[r, "\widetilde{D}_k"] & 
						Q_{\tau}(M)
					\end{tikzcd}
				\end{equation}
				It is easy to note that the first two squares in diagram \eqref{eq4.12} commute. It is now enough to show that the following equality holds for all $k=0,1,\ldots,n$. 
				\begin{equation}
					\widetilde{D}_k\circ \Phi_M \circ g\circ g_I=f_{Q_{\tau}(M)}\circ (T\widetilde{D}_k+T\widetilde{D}_{k-1}\circ(\Delta_1)_{Q_{\tau}(M)}+\dots+(\Delta_k)_{Q_\tau(M)})\circ T\Phi_M\circ Tg.
				\end{equation}
				We note that $f_{Q_{\tau}(M)}\circ T\Phi_M=\Phi_M\circ f_M$ and $(\Delta_k)_{Q_{\tau}(M)}\circ T\Phi_M=T\Phi_M\circ (\Delta_k)_{M}$ for all $k=0,1,\ldots,n$. Further, from Lemma \ref{K2}, we have $\widetilde{D}_k\circ\Phi_M=\Phi_M \circ {D}_k$  for all $k=0,1,\ldots,n$. Therefore, we get
				\begin{align*}
					&f_{Q_{\tau}(M)}\circ (T\widetilde{D}_k+T\widetilde{D}_{k-1}\circ(\Delta_1)_{Q_{\tau}(M)}+...+(\Delta_k)_{Q_\tau(M)})\circ T\Phi_M\circ Tg \\
					&=\sum_{t=0}^kf_{Q_{\tau}(M)}\circ T\widetilde{D}_t\circ (\Delta_{k-t})_{Q_{\tau M}}\circ T\Phi_M \circ Tg\\
					&=\sum_{t=0}^kf_{Q_{\tau}(M)}\circ T\widetilde{D}_t\circ T\Phi_M\circ (\Delta_{k-t})_M\circ Tg\tag{$(\Delta_{k-t})_{Q_{\tau}(M)}\circ T\Phi_M=T\Phi_M\circ (\Delta_{k-t})_{M}$}\\
                    &=\sum_{t=0}^kf_{Q_{\tau}(M)}\circ T\Phi_M\circ T{D}_t\circ  (\Delta_{k-t})_M\circ Tg\tag{ $\widetilde{D}_k\circ\Phi_M=\Phi_M \circ {D}_k$}\\
					&=\sum_{t=0}^k \Phi_M\circ f_M\circ T{D}_t\circ  (\Delta_{k-t})_M\circ Tg\tag{$f_{Q_{\tau}(M)}\circ T\Phi_M=\Phi_M\circ f_M$}\\
					&=\Phi_M \circ f_M\circ (TD_k+TD_{k-1}\circ(\Delta_1)_M+\ldots+(\Delta_k)_M)\circ Tg \\
					&=\Phi_M\circ D_k \circ f_M\circ Tg\tag{$D$ is a higher derivation of order $n$ }\\
					&=\widetilde{D}_k\circ \Phi_M\circ f_M\circ Tg\tag{ $\widetilde{D}_k\circ\Phi_M=\Phi_M \circ {D}_k$}\\
					&=\widetilde{D}_k\circ \Phi_M\circ g\circ g_I\tag{$ f_M\circ Tg=g\circ g_I$}.
				\end{align*}The proof is now complete.
			\end{proof}
			By Lemma \ref{K2} and Theorem \ref{K},
			$\widetilde{D}$ is a higher $\Delta$-derivation of order $n$ on $Q_\tau(M)$, 
			which extends the higher derivation $D$ of order $n$ on $M$. In the following theorem, we show that 
			this lifting is uniquely determined.
			\begin{Thm}[Uniqueness of Extension of Higher $\Delta$-Derivation]
			    Let $(T,\theta,\zeta,\Delta)$ be a higher differential monad of order $n$ on $\mathscr C$ that is exact and preserves colimits and $\tau$ be a hereditary higher differential torsion theory of order $n$ on $EM_T$. Let $D: M \longrightarrow M$ be a higher $\Delta$-derivation of order $n$ on a torsion-free module $(M, f_M) \in EM_T$. Then the higher $\Delta$-derivation $\widetilde{D} : Q_{\tau}(M) \longrightarrow Q_{\tau}(M)$ of order $n$, obtained in Theorem (\ref{K}), is the unique higher $\Delta$-derivation of order $n$ on $Q_{\tau}(M)$, extending $D$. That is, $\widetilde{D}$ is the only higher $\Delta$-derivation of order $n$ satisfying Lemma \ref{K2}, i.e, \[
				\widetilde{D}_k \circ \Phi_M = \Phi_M \circ D_k
				\]  for all $0\leq k\leq n.$
			\end{Thm}
			\begin{proof}
				Let $\widetilde{D'}$ be another higher $\Delta$-derivation of order $n$ on $Q_{\tau}(M)$ that lifts the $\Delta$-derivation $D$ on $M$. We now proceed by induction on $n.$ For $n=0$, we have $\widetilde{D}_0 = \widetilde{D'}_0 = \operatorname{Id}$, and hence the statement holds for $n=0$.
				We now assume that the statement is true for all $k\leq n-1$. For $k=n$, we see that
				\begin{align*}
					(\widetilde{D}_n-\widetilde{D'}_n) \circ f_{Q_{\tau}(M)}
					&= f_{Q_{\tau}(M)} \circ \big( T\widetilde{D}_n + T\widetilde{D}_{n-1} \circ(\Delta_1)_{Q_{\tau}(M)} + T\widetilde{D}_{n-2} \circ(\Delta_2)_{Q_{\tau}(M)} + \ldots + \Delta_{n_{Q_{\tau}(M)}} \big) \\
					&\quad - f_{Q_{\tau}(M)} \circ \big( T\widetilde{D'}_n + T\widetilde{D'}_{n-1} \circ(\Delta_1)_{Q_{\tau}(M)} + T\widetilde{D'}_{n-2}\circ (\Delta_2)_{Q_{\tau}(M)} + \ldots + (\Delta_n)_{Q_{\tau}(M)} \big) \\
					&= f_{Q_{\tau}(M)} \circ T(\widetilde{D_n} - \widetilde{D_n'}).\tag{since $\widetilde{D}_k=\widetilde{D'}_k$ for all $k\leq n-1$}
				\end{align*}
				Hence $\widetilde{D}_n - \widetilde{D'}_n$ is a morphism on $Q_{\tau}(M)$ in $EM_T$. Since 
				\[
				(\widetilde{D}_n - \widetilde{D'}_n) \circ \Phi_M = \Phi_M \circ (D_n - D_n) = 0,
				\]
				there exists an induced morphism $\operatorname{Coker}(\Phi_M) \longrightarrow Q_{\tau}(M)$ in $EM_T$ such that $\widetilde{D}_n - \widetilde{D'}_n$ factors through it. By \cite[Proposition~2.4 and Theorem~2.5]{GG}, we have $\operatorname{Coker}(\Phi_M)\in \mathscr{T}$ and $Q_\tau(M)\in \mathscr{F}$. Therefore, the morphism $\operatorname{Coker}(\Phi_M)\longrightarrow Q_\tau(M)$ is the zero morphism and hence $\widetilde{D}_n - \widetilde{D'}_n = 0$, i.e., $\widetilde{D}_k=\widetilde{D'}_k$ for all $n$.
			\end{proof}
			\begin{Thm}\label{T5.5}
				Let $(T,\theta,\zeta,\Delta)$ be a higher differential monad of order $n$ on $\mathscr{C}$ that is exact and preserve colimits. Let $\tau$ be a hereditary torsion theory on $EM_T$ and $(M,f_M)$ be any object in $EM_T$. Then any higher $\Delta$-derivation $D: M\longrightarrow M$ of order $n$ can be extended to a unique higher $\Delta$-derivation $\widetilde{D}:Q_\tau (M)\longrightarrow Q_\tau (M)$ of order $n$ if and only if $\tau $ is a higher differential torsion theory of order n.
			\end{Thm}
			\begin{proof}
				The sufficient part of this theorem is straightforward and follows by employing the same arguments and techniques as those presented in \cite[Theorem 4.5]{DA}. Therefore, we omit the details and refer the reader to the corresponding proof in the cited work. To prove the necessary part, we proceed with the higher $\Delta$-derivation $D$ on $M$ and its extension $\widetilde{D}$ of order $n$ on $Q_{\tau}(M)$. For  all $i$, we now consider the following commutative diagram
				\begin{equation}\label{eq3.5}
					\begin{tikzcd}[row sep=3.8em, column sep = 7em]
						M \arrow{r}{\Phi_M }\arrow{d}{D_i} & Q_{\tau}(M) \arrow{d}{\widetilde{D_i}}\\
						M \arrow{r}{\Phi_M}& Q_{\tau}(M)~.
					\end{tikzcd}
				\end{equation}
				 Let $N\subseteq M_\tau$. Since $M_{\tau}=\ker(\Phi_M)$,  $\Phi_M\circ D_i(N)=\widetilde{D}_i\circ \Phi_M(N)=0$ for all $i$. Therefore $D_i(N)\subseteq \text{ ker }(\Phi_M)=M_\tau $ for all $i$. Hence $\tau $ is a higher differential torsion theory of order n.
			\end{proof}

		\end{document}